\documentclass[11pt,letterpaper]{article}
\usepackage[utf8]{inputenc}

\usepackage[left=1in, right=1in, top=1in, bottom=1in]{geometry}
\usepackage{amsmath,amsfonts,amssymb,amsthm,amsopn,amstext,amscd}
\usepackage{natbib}
\usepackage[english]{babel}
\usepackage{graphicx}
\usepackage{caption}
\usepackage{subcaption}
\usepackage{epstopdf}
\usepackage{placeins}
\usepackage{multirow}
\usepackage{hyperref}

\newtheorem{definition}{Definition}
\newtheorem{proposition}{Proposition}

\begin{document}
	
	\author{Camilo Jose Torres-Jimenez\footnote{Corresponding author. Candidate, Ph.D. Program in Statistics. Universidad Nacional de Colombia - Sede Bogota, Colombia. Email address: \texttt{cjtorresj@unal.edu.co}. Postal address: Carrera 109 No. 136 A - 36, Bogotá, Colombia.}
	\and
	Alvaro Mauricio Montenegro-Diaz\footnote{Associate Professor. Department of Statistics. Universidad Nacional de Colombia - Sede Bogota, Colombia.}}
	
	\title{An alternative to continuous univariate distributions supported on a bounded interval: The BMT distribution}
	
	\maketitle
	
	\begin{abstract}
		In this paper, we introduce the BMT distribution as an unimodal alternative to continuous univariate distributions supported on a bounded interval.
		The ideas behind the mathematical formulation of this new distribution come from computer aid geometric design, specifically from Bezier curves.
		First, we review general properties of a distribution given by parametric equations and extend the definition of a Bezier distribution.
		Then, after proposing the BMT cumulative distribution function, we derive its probability density function and a closed-form expression for quantile function, median, interquartile range, mode, and moments.
		The domain change from [0,1] to [c,d] is mentioned.
		Estimation of parameters is approached by the methods of maximum likelihood and maximum product of spacing.
		We test the numerical estimation procedures using some simulated data. 
		Usefulness and flexibility of the new distribution are illustrated in three real data sets.
		The BMT distribution has a significant potential to estimate domain parameters and to model data outside the scope of the beta or similar distributions.
	\end{abstract}
	
	
	
	\section{Introduction}
	
	The well known beta distribution is frequently the first choice for doubled-bounded data. Additionally, the Kumaraswamy distribution is a viable alternative with some common characteristics and also some advantages over the beta distribution \citep{jones2009kumaraswamy}. On the downside, the Kumaraswamy does not have exact symmetric shapes or close-form expressions for its moments. Yet, some reparametrizations using a simple formula for the median help modeling purposes \citep{mitnik2013kumaraswamy}. 
	 
	Recently, there has been a bloom of new distributions that arise combining existing ones. For example, \citet{cordeiro2011new} stated a methodology to produce new distributions using the Kumaraswamy. At first sight that methodology is most likely extensible for any distribution on $[0,1]$, providing as many distributional families as possible combinations. 
	
	On the other hand, we have the generalizations. Some of them seek more shapes for a distribution family, e.g. the generalized beta \citep{mcdonald1995generalization}, and other go after a unique formula or expression that brings together several known distributions, e.g. the Johnson translation system \cite[Section 4.3]{johnson1996continuous}. Unfortunately in both cases, the number of parameters increases, and not all of them have a recognizable interpretation.
	
	Having so many options, \citet{jones2015families} did a widespread review and comparison of main general techniques for providing mostly unimodal distribution families on $\mathbb{R}$. The reviewed distributions usually have one, two, or even three shape parameters controlling skewness and/or tail-weight, in addition to their location and scale parameters. \citet{jones2015families} identifies four different construction techniques: Family 1: Azzalini-Type Skew-Symmetric Distributions; Family 2: Transformation of Random Variable; Family 3: Transformation of Scale, including Family 3A: Two-Piece; and, Family 4: Probability Integral Transformation of $(0, 1)$ Random Variable. In that review, it is also mentioned that all these construction techniques can be redefined to obtain distributions with a bounded support. In our opinion, Jones' most important conclusion about so many existent distributions is that: ``The ongoing challenge is to extract from the overwhelming plethora of possibilities those relatively few with the best and most appropriate properties that are of real potential value in practical applications.''
	
	In this work, we propose a parametric distribution family looking for a useful alternative to existing ones supported on a bounded interval. Our distribution brings something new over the existing options, and it has great potential for practical applications. 
	
	This new distribution was obtained outside the techniques mentioned by \citet{jones2015families}. It originates from Bezier curves, parametric equations frequently used in graphic computation. The Bezier curves are very flexible and important in computer design and engineering modeling. By establishing some conditions, the parametric equations of Bezier curves fulfill the requirements of a cumulative distribution function. Hence, those curves can provide diverse shapes as a distribution family. \citet{wagner1996using} proposed a distribution based on Bezier curves to model the input of engineering processes. After that, Bezier curves have rarely appeared in the probability and statistics area \citep{wagner1995graphical,wagner1996recent,kim1996nonparametric,kim1999smoothing,kim2000simulation,kim2003bezier,kotz2004other,kuhl2010univariate,kim2012general,bae2014nonparametric,cha2016bezier}. 
	
	Initially, without knowledge of the work of \citet{wagner1996using}, we put together a cumulative distribution function based on Bezier curves. Our main motivations Bezier curves parametric equations are their capacity of being molded to assume a desired form and their useful mathematical properties. We name the resultant distribution family BMT, as an acronym for Bezier-Montenegro-Torres. The BMT distribution has the following features: finite support, small number of interpretable parameters, symmetric and skewed unimodal shapes, and close-form formulas for quantile function (therefore median and interquartile range), mode, and moments.

	After establishing some mathematical characteristics of the BMT distribution, we study two different estimation methods for the parameters of the distribution. Maximum likelihood and maximum product of spacing are tested, given us satisfactory numerical outcomes, at least with the simulations and applications considered.
	
	As a result, we obtain two interesting and important things about the proposed distribution family. 
	First, there are values of (population) skewness and kurtosis that are possible for the BMT distribution, while said values are not reachable with the beta and Kumaraswamy distribution families. Thus, if we use as a criterion the sample skewness and kurtosis, there are datasets that should be modeled with the BMT instead of the beta and Kumaraswamy distributions.
	Second, the 4-parameter BMT distribution works much better than the equivalent 4-parameter version of the beta and Kumaraswamy distributions. That occurs because estimates for domain parameters under the BMT distribution come considerably closer to the minimum and maximum of the sample than the mentioned competing distributions. 
	
	The rest of the paper is structured as follows: In Section \ref{sec:background}, some characteristics of Bezier curves, generic distributions given by parametric equations, and the Bezier distribution are presented. In Section \ref{sec:cdf_pdf}, cumulative distribution and probability density functions for the BMT distribution are characterized and described. In Section \ref{sec:quantile}, quantile function, median, interquartile range and a sampling procedure are established. In Section \ref{sec:moments}, moments of a BMT random variable are considered. In Section \ref{sec:BMT4}, the BMT distribution is extended from $[0,1]$ to $[c,d]$. In Section \ref{sec:estimation}, two estimation methods, maximum likelihood and maximum product of spacing, are reviewed. In Section \ref{sec:applications}, the potential and usefulness of the BMT distribution are depicted through the fitting of three real data sets. In Section \ref{sec:conclusions}, concluding remarks, observations, and future work are addressed.
	
	\section{Background} \label{sec:background}
	
	In this section, we point out properties and results about Bezier curves, generic distributions given by parametric equations, and the Bezier distribution. 
	
	\subsection{Bezier curves}
	
	In computer graphics, a Bezier curve \citep{bezier1977essai} is used to approximate smooth shapes, especially for the computer aided geometric design (CAGD) \citep{farin2002curves}. The curve is represented by parametric equations given by polynomials, which can be expressed in different ways: the Bernstein form, the de Casteljau's algorithm (recursive form), the polynomial form, and the matrix form. For example, the Bernstein form of a Bezier curve is given by, 
	\begin{equation}  \label{eq:Bezier}
		\mathrm{\mathbf{b}}^n(t) = \sum_{i=0}^n \mathbf{b}_i \, \mathrm{B}_i^n(t), 
	\end{equation}	
	where $t \in [0,1]$, $\mathbf{b}_0, \mathbf{b}_1, \dots, \mathbf{b}_n \in \mathbb{R}^d$ are the Bezier control points, and $\{\mathrm{B}_i^n(t), i=0,\dots,n\}$ are the $n+1$ Bernstein basis polynomials of degree $n$ \cite[Section 5.1]{farin2002curves}. Each Bernstein polynomial is defined explicitly by $\mathrm{B}_i^n(t) = \binom{n}{i} t^i (1-t)^{n-i}$.
	
	Some properties of a Bezier curve are: affine invariance, invariance under affine parameter transformations, convex hull property, endpoint interpolation, symmetry with respect to $t$ and $1-t$, invariance under baricentric combinations, linear precision, and pseudolocal control \cite[Section 5.2]{farin2002curves}. 
	
	Affine invariance is an important property of Bezier curves. It means that they are invariant under affine maps. Some examples of affine maps are translations, scalings, rotations, shears, and parallel projections. 
	
	Another useful property of Bezier curves is the closed-form expressions for their derivatives. From the Bernstein form, the $r$-th derivative of a Bezier curve is given by, 
	\begin{equation} \label{eq:dBezier}
	\frac{d^r}{\, dt^r} \mathrm{\mathbf{b}}^n(t) = \frac{n!}{(n-r)!} \sum_{i=0}^{n-r} \Delta^r \mathbf{b}_{i} \, \mathrm{B}_{i}^{n-r}(t),
	\end{equation}
	where $t \in [0,1]$, and $\Delta^r\mathbf{b}_{j}=\sum_{i=0}^r \binom{r}{i} \, (-1)^{r-i} \, \mathbf{b}_{i+j}$ \cite[Section 5.3]{farin2002curves}.
	
	\subsection{Any distribution given by parametric equations}
	
	The Bezier distribution proposed by \citet{wagner1996using} shows that continuous distributions given by parametric equations have already been worked in the literature. However, we did not find a summary of general properties of such distributions in our bibliographic review.
	
	If a curve given by parametric equations, $x = \mathrm{x}(t)$ and $y_\mathrm{F} = \mathrm{y}_\mathrm{F}(t)$, fulfills the conditions of a cumulative distribution function (CDF), then there is a random variable $X$ with CDF $\mathrm{F}_X$ described by that curve and,
	\[\mathrm{F}_X(x) = \mathrm{y}_\mathrm{F} \left( \mathrm{x}^{-1}(x) \right).\]
	
	It follows that, if it exists, the probability density function (PDF) of $X$ can be given by the parametric equations, 
	\[x = \mathrm{x}(t) \quad \text{ and } \quad y_\mathrm{f} = \mathrm{y}_\mathrm{f}(t) = \frac{\mathrm{y}'_\mathrm{F}(t)}{\mathrm{x}'(t)},\] 
	or by the function,
	\[\mathrm{f}_X(x) = \frac{\mathrm{y}'_\mathrm{F} \left( \mathrm{x}^{-1}(x) \right) }{\mathrm{x}' \left( \mathrm{x}^{-1}(x) \right) }.\] 
	
	The quantile function is given by the parametric equations, $p = \mathrm{y}_\mathrm{F}(t)$ and $y_\mathrm{Q} = \mathrm{x}(t)$, i.e., 
	\[\mathrm{F}^{-1}_X(p) = \mathrm{x} \left( \mathrm{y}_\mathrm{F}^{-1}(p) \right).\] 
	
	The $r$-th central moment, for $r \in \mathbb{Z}^+$, and the characteristic function of $X$ are: 
	\[\mu_X^{(r)} = \int \big( \mathrm{x}(t) - \mu \big)^r \mathrm{y}'_\mathrm{F}(t) \, dt \quad \text{ and } \quad \psi_X(s) = \int \exp \left(i s \mathrm{x}(t)\right) \mathrm{y}'_\mathrm{F}(t) \, dt,\]
	respectively. 
	
	\subsection{Bezier distribution}
	
	\citet{wagner1996using} propose the following definition for the Bezier distribution.
	
	\begin{definition}\label{def:BezierDist1}
		If $X$ is a Bezier continuous random variable with bounded support $[x_*,x^*]$, then the CDF of $X$ is given parametrically by,
		\begin{align}
			\mathbf{b}^n(t) &= \big( \mathrm{b}^n_1(t), \mathrm{b}^n_2(t) \big)^T = \big( \mathrm{x}(t) , \mathrm{F}_X \left( \mathrm{x}(t) \right) \big)^T ,
		\end{align}
		where $t \in [0,1]$ and the Bezier control points $\mathbf{b}_0, \mathbf{b}_1, \dots \mathbf{b}_n \in \mathbb{R}^2$ fulfill,
		\begin{enumerate}
			\item[(i)] $\mathbf{b}_0 = (b_{0,1},b_{0,2})^T = (x_*,0)^T$ and $\mathbf{b}_n = (b_{n,1},b_{n,2})^T = (x^*,1)^T$,
			\item[(ii)] $x_* \leq b_{1,1} \leq \dots \leq b_{n-1,1} \leq x^*$ and $0 \leq b_{1,2} \leq \dots \leq b_{n-1,2} \leq 1$.
		\end{enumerate}
	\end{definition}

	With Definition \ref{def:BezierDist1}, the Bezier distribution family is only a subset of all the CDFs that could be described by a Bezier curve. We propose the following, more general, definition.
	
	\begin{definition}\emph{(Bezier distribution)}\label{def:BezierDist2}
		$X$ is a Bezier random variable, when the CDF of $X$ is given by the parametric equations,
		\begin{subequations} \label{eq:EBezier}
			\begin{equation} \label{eq:EBezierX}
			\mathrm{x}(t) = \mathrm{b}^n_1(t) I_{[0,1]}(t) + \left((b_{n,1} - b_{0,1}) t + b_{0,1}\right) I_{(-\infty,0)\cup(1,\infty)}(t),
			\end{equation}
			\begin{equation}  \label{eq:EBezierY}
			\mathrm{y}_\mathrm{F}(t) = \mathrm{F}_X \left( \mathrm{x}(t) \right) = \mathrm{b}^n_2(t) \mathrm{I}_{[0,1)} \big( t \big) + \mathrm{I}_{[1,\infty)} \big( t \big),
			\end{equation} 
		\end{subequations}
		where $t \in \mathbb{R}$ and $\mathrm{\mathbf{b}}^n(t) = \big( \mathrm{b}^n_1(t), \mathrm{b}^n_2(t) \big)^{T}$ is a Bezier curve, with Bezier control points $\mathbf{b}_0 = \left(b_{0,1},b_{0,2}\right)^{T}, \mathbf{b}_1 = \left(b_{1,1},b_{1,2}\right)^{T}, \dots, \mathbf{b}_n = \left(b_{n,1},b_{n,2}\right)^{T} \in \mathbb{R}^2$ that fulfill the following conditions: 
		\begin{enumerate}
			\item[(i)] $0 \leq b_{0,2}$ and $b_{n,2} \leq 1$,
			\item[(ii)] $\sum_{i=0}^{n-1} \left( b_{i+1,1} - b_{i,1} \right) \mathrm{B}_i^{n-1}(t) \geq 0$ and $\sum_{i=0}^{n-1} \left( b_{i+1,2} - b_{i,2} \right) \mathrm{B}_i^{n-1}(t) \geq 0$.
		\end{enumerate}
	\end{definition}
	
	Figure \ref{fig:CubicBeziercdfpdf} give us an example of a Bezier distribution under Definition \ref{def:BezierDist2}. This example distribution is supported on $(b_{0,1},b_{3,1})^{T}$. Its associated Bezier curve has four control points (so it is a cubic Bezier curve), and the image of that curve is $(b_{0,2},b_{3,2})^{T} \subset [0,1]$. Also, attraction points $\mathbf{b}_1$ and $\mathbf{b}_2$ determine the slope of the CDF tangent lines at $\mathbf{b}_0$ and $\mathbf{b}_3$, respectively.
	
	\begin{figure}[t]
		\centering
		\begin{subfigure}{0.475\textwidth}
			\includegraphics[width=\textwidth]{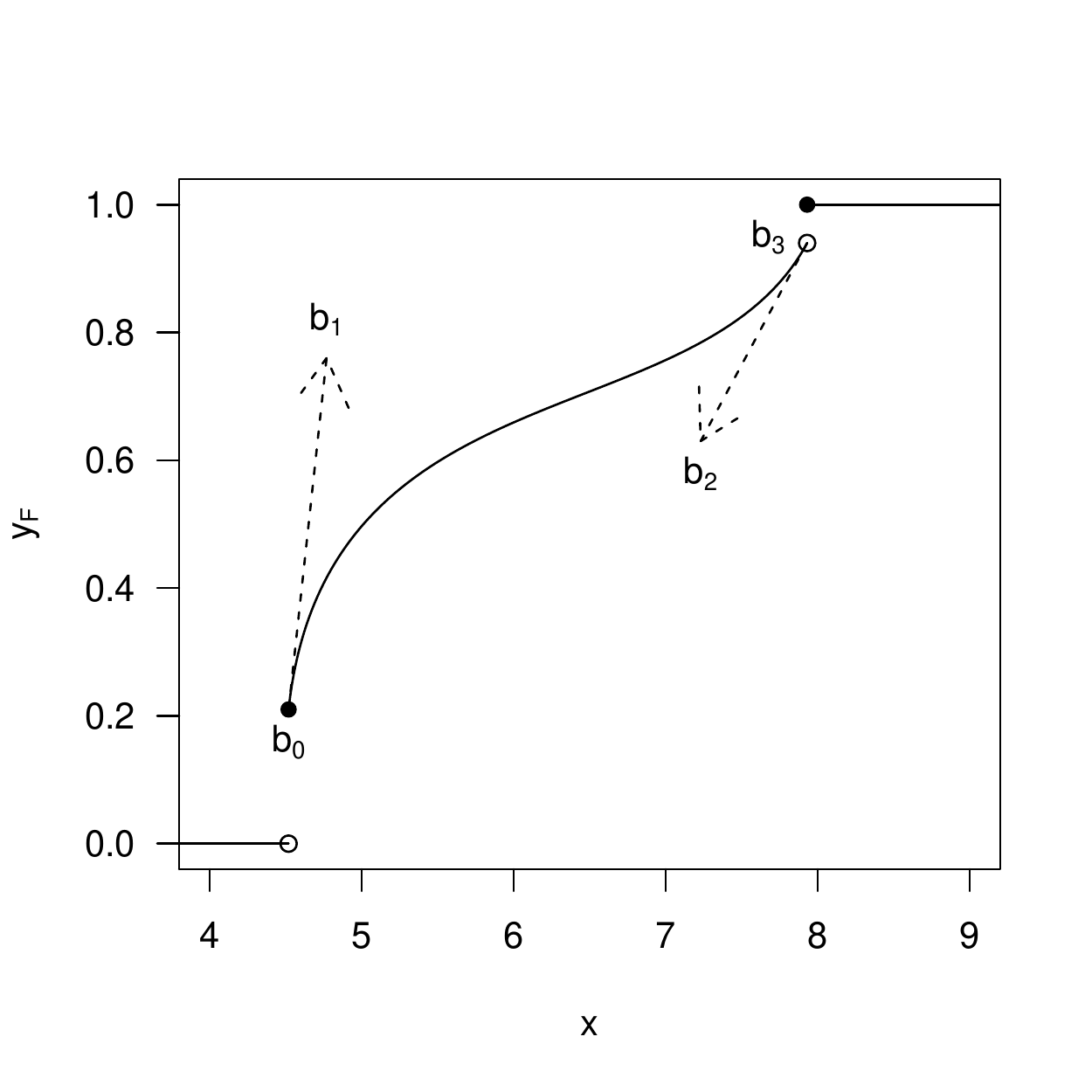}
			\caption{Cumulative distribution function.}
			\label{fig:CubicBeziercdf}
		\end{subfigure}
		\begin{subfigure}{0.475\textwidth}
			\includegraphics[width=\textwidth]{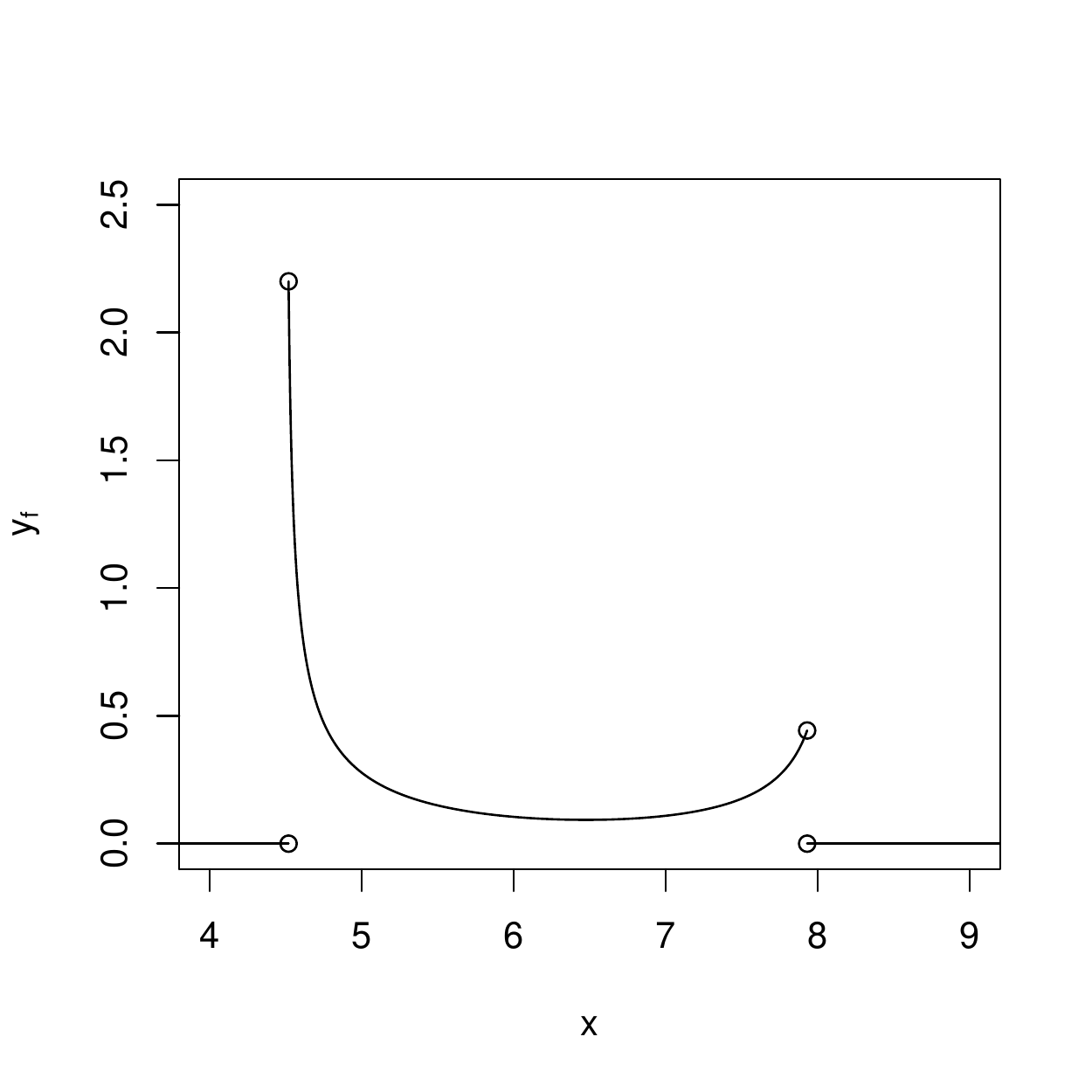}
			\caption{Probability density function.}
			\label{fig:CubicBezierpdf}
		\end{subfigure}
		\caption{Cumulative distribution and probability density functions for a Bezier distribution with control points $\mathbf{b}_0$, $\mathbf{b}_1$, $\mathbf{b}_2$, and $\mathbf{b}_3$.}
		\label{fig:CubicBeziercdfpdf}
	\end{figure}
	
	In fact, Definition~\ref{def:BezierDist2} includes all the CDFs that can be given by the parametric equations of a Bezier curve: Equations (\ref{eq:EBezierX}) and (\ref{eq:EBezierY}) guaranties that not only $\lim_{x \to -\infty} \mathrm{F}_X(x)=0$ and $\lim_{x \to \infty} \mathrm{F}_X(x)=1$, but also right-continuity of $\mathrm{F}_X$ (because Bezier curves are already continuous from point $\mathbf{b}_0$ to point $\mathbf{b}_n$). Condition (i) restricts the codomain of $\mathrm{F}_X$ to the interval $[0,1]$, as it is required by a CDF, and, at the same time, it includes distributions with $\lim_{h \to 0} \mathrm{F}_X\left(b_{0,1}+h\right) \neq 0$ and $\lim_{h \to 0} \mathrm{F}_X\left(b_{n,1}-h\right) \neq 1$. Finally, $\mathrm{F}_X$ has to be a non-decreasing function to be a CDF. Using the first derivative of a Bezier curve (from Equation (\ref{eq:dBezier})), it can be shown that $\mathrm{F}_X$ is a non-decreasing function if and only if condition (ii) is fulfilled (see Appendix~\ref{ap:cond2BezierDef}).
	
	\citet{bae2014nonparametric} propose the Bezier smoothing as a non-parametric technique to estimate a CDF. Therefore, all their estimated CDFs are Bezier distributions under Definition \ref{def:BezierDist2}.
		
	Since Definition \ref{def:BezierDist1} and \ref{def:BezierDist2} differ only in the discontinuities allowed by Equations (\ref{eq:EBezier}) and in conditions for Bezier control points, the pdf, moments, and generation of random values are practically the same as the ones presented by \citet{wagner1996using}. From now on, we will assume that the Bezier distribution family refers to Definition \ref{def:BezierDist2}.
	
	Given that the Bezier distribution inherits the affine invariance of Bezier curves, then we have the following property, which has some resemblance to the location-scale property of some well-known distributions.
	
	\begin{proposition}\label{prop:EBezierlinealescalar}
		If $X$ is a Bezier random variable given by the Bezier control points $\mathbf{b}_0, \mathbf{b}_1, \dots, \mathbf{b}_n \in \mathbb{R}^2$, and if $Z = u \, X + v$ where $u \in \mathbb{R}^+$ and $v \in \mathbb{R}$. Then $Z$ is a Bezier random variable given by the Bezier control points $\mathbf{b}_0^{*}, \mathbf{b}_1^{*}, \dots, \mathbf{b}_n^{*} \in \mathbb{R}^2$ where $\mathbf{b}_i^{*} = \left(u \, b_{i,1} + v, b_{i,2} \right)^\mathrm{T}$ for $i = 0, \dots, n$.
	\end{proposition}
	
	\begin{proof}
		$Z = u \, X + v$ is an affine map, as a scaling and translation of $X$. That affine map will be applied to the Bezier curve associated to the CDF of $X$. Also, affine invariance of Bezier curves means that applying the affine map to a Bezier curve or to its control points leads us to the same result. Then, the scaling and translation of $X$ is equivalent to the scaling and translation of the Bezier control points related to the CDF of $X$.
	\end{proof}
	
	On the other hand, \citet{wagner1996using} propose to numerically compute the moments of a Bezier random variable. To that end, they use Gaussian quadrature and the following result for a nonnegative random variable $X$,
	\[ E\left[ X^r \right] = \int_{0}^{1} r \left(\mathrm{x}(t)\right)^{r-1} \left(1-\mathrm{y}_\mathrm{F}(t)\right) \left|\mathrm{x}'(t)\right| \, dt \]
	
	Nevertheless, we obtained some expressions for the raw and central moments of $X$. In Appendix \ref{ap:rawMBezier}, we prove that the $r$-th raw moment of a Bezier random variable $X$ is given by, 
	\begin{equation} \label{eq:BezierRawMoment}
	E\left[ X^r \right] = \mu_{\;X}^{*(r)} = \frac{r!}{r+1} \sum_{k_0+\dots+k_n=r} \sum_{j=0}^{n-1} \frac{\left( \prod_{i=0}^n \binom{n}{i}^{k_i} b_{i,1}^{k_i} \right) \binom{n-1}{j} \left( b_{j+1,2} - b_{j,2} \right) }{k_0! \dots k_n! \, \binom{(r+1)n-1}{j+\sum_{i=0}^n ik_i}}.
	\end{equation}
	From Equation (\ref{eq:BezierRawMoment}) and Proposition~\ref{prop:EBezierlinealescalar}, we get that the $r$-th central moment of $X$ is given by,
	\begin{equation} \label{eq:BezierCentralMoment}
	E\left[ \left(X - \mu_X\right)^r \right] = \mu_{X}^{(r)} = \frac{r!}{r+1} \sum_{k_0+\dots+k_n=r} \sum_{j=0}^{n-1} \frac{\left( \prod_{i=0}^n \binom{n}{i}^{k_i} \left(b_{i,1} - \mu_X\right)^{k_i} \right) \binom{n-1}{j} \left( b_{j+1,2} - b_{j,2} \right)}{k_0! \dots k_n! \, \binom{(r+1)n-1}{j+\sum_{i=0}^n i \, k_i}}.
	\end{equation}
	
	\section{Cumulative distribution and probability density functions} \label{sec:cdf_pdf}
	
	The BMT is a parametric family of continuous probability distributions supported on the interval $[0,1]$ and characterized by only two parameters, that we denoted $\kappa_l$ and $\kappa_r$, both of them on the interval $(0,1)$. These shape parameters control the curvature of each tail, $\kappa_l$ for the left tail and $\kappa_r$ for the right one.
	
	This distribution family was obtained after set some specifications for the Bezier distribution with four control points. Thus, it benefits of Bezier curves and Bezier distribution properties. Its Bezier control points are $\mathbf{b}_0=\left(0,0\right)^\mathrm{T}, \mathbf{b}_1=\left(\kappa_l,0\right)^\mathrm{T}, \mathbf{b}_2=\left(1-\kappa_r,1\right)^\mathrm{T}$, and $\mathbf{b}_3=\left(1,1\right)^\mathrm{T}$. Hence, the BMT distribution is supported on $[0,1]$ because $b_{0,1} = 0$ and $b_{3,1} = 1$. All BMT CDFs are continuous on $\mathbb{R}$ because $b_{0,2} = 0$ and $b_{3,2} = 1$. And, since $b_{1,2} = b_{0,2} = 0$ and $b_{2,2} = b_{3,2} = 1$, all BMT PDFs are continuous at $0$ and $1$.
	
	From Definition \ref{def:BezierDist2} and the mentioned control points, $\mathbf{b}_0$ to $\mathbf{b}_3$, we have that the CDF of a BMT random variable $X$ is given parametrically by,
				\begin{align*}
				\mathrm{x}(t) &= \left(b_{0,1} \mathrm{B}_0^3(t) + b_{1,1} \mathrm{B}_1^3(t) + b_{2,1} \mathrm{B}_2^3(t) + b_{3,1} \mathrm{B}_3^3(t)\right) I_{[0,1]}(t) + \left(\left(b_{3,1} - b_{0,1}\right)t + b_{0,1}\right) I_{(-\infty,0)\cup(1,\infty)}(t), \\
				&= \left((0) (1-t)^3 + (\kappa_l) 3 t (1-t)^2 + (1-\kappa_r) 3 t^2 (1-t) + (1) t^3\right) I_{[0,1]}(t) + t \, I_{(-\infty,0)\cup(1,\infty)}(t), \\
				\mathrm{y}_\mathrm{F}(t) &= \left( b_{0,2} \mathrm{B}_0^3(t) + b_{1,2} \mathrm{B}_1^3(t) + b_{2,2} \mathrm{B}_2^3(t) + b_{3,2} \mathrm{B}_3^3(t) \right) \mathrm{I}_{[0,1)} ( t ) + \mathrm{I}_{[1,\infty)} ( t ), \\
				&= \left( (0) (1-t)^3 + (0) 3 t (1-t)^2 + (1) 3 t^2 (1-t) + (1) t^3 \right) \mathrm{I}_{[0,1)} ( t ) + \mathrm{I}_{[1,\infty)} ( t ),
				\end{align*} 
	and rewriting those polynomials with respect to $t$, we present the following definition.
	
	\begin{definition}\emph{(BMT distribution)}. \label{def:BMT}
		A random variable $X$ is said to be BMT distributed, denoted by $BMT \left( \kappa_l, \kappa_r \right)$, if its CDF is given by the following parametric equations,
		\begin{subequations} \label{eq:BMT}
			\begin{equation} \label{eq:BMTX}
			\mathrm{x}(t) = \left(\left( 3\kappa_{l} + 3\kappa_r - 2\right) t^{3} + \left( - 6\kappa_{l} - 3\kappa_r + 3 \right) t^{2} + (3\kappa _{l}) t\right) I_{[0,1]}(t) + t \, I_{(-\infty,0)\cup(1,\infty)}(t),
			\end{equation}
			\begin{equation} \label{eq:BMTY}
			\mathrm{y}_\mathrm{F}(t) = \mathrm{F}_X \left(\mathrm{x}(t)\right) = \left( - 2 t^{3} + 3 t^{2} \right) I_{[0,1)}(\mathrm{x}(t)) + I_{[1, \infty)}(\mathrm{x}(t)) ,
			\end{equation}
		\end{subequations}
		for $\kappa_l, \kappa_r \in (0,1)$ and $t \in \mathbb{R}$.
	\end{definition}
	
	If $X$ is a BMT random variable with CDF $\mathrm{F}_X$ given by parametric equations (\ref{eq:BMTX}) and (\ref{eq:BMTY}), the correspondent PDF $\mathrm{f}_X$ is given parametrically by,
	\begin{subequations} \label{eq:BMTpdf}
		\begin{equation} \label{eq:BMTXpdf} 
		\mathrm{x}(t) = \left(\left( 3\kappa_{l} + 3\kappa_r - 2\right) t^{3} + \left( - 6\kappa_{l} - 3\kappa_r + 3 \right) t^{2} + (3\kappa _{l}) t\right) I_{[0,1]}(t) + t \, I_{(-\infty,0)\cup(1,\infty)}(t),
		\end{equation}
		\begin{equation} \label{eq:BMTYpdf}
		\mathrm{y}_\mathrm{f}(t) = \mathrm{f}_X\left(\mathrm{x}(t)\right) = \left( \frac{2t(1-t)}
		{ \left( 3\kappa _{l}+3\kappa_r-2\right) t^{2} + \left( -4\kappa _{l}-2\kappa_r+2\right) t + \kappa _{l} } \right) I_{(0,1)}(t) 
		\end{equation}
	\end{subequations}
	for $\kappa_l, \kappa_r \in (0,1)$ and $t \in \mathbb{R}$.
	
	Figure \ref{fig:BMTcdfspdfs} illustrates BMT CDFs and PDFs for different values of $\kappa_l$ and $\kappa_r$, including limiting cases in which those parameters tend to $0$ or $1$. Each cell of the shape plot in Figure \ref{fig:BMTcdfs} represents the square $[0,1] \times [0,1]$, whereas each cell of the shape plot in Figure \ref{fig:BMTpdfs} represents the square $[0,1] \times [0,6]$. 
	
	\begin{figure}[t]
		\centering
		\begin{subfigure}{0.475\textwidth}
			\includegraphics[width=\textwidth]{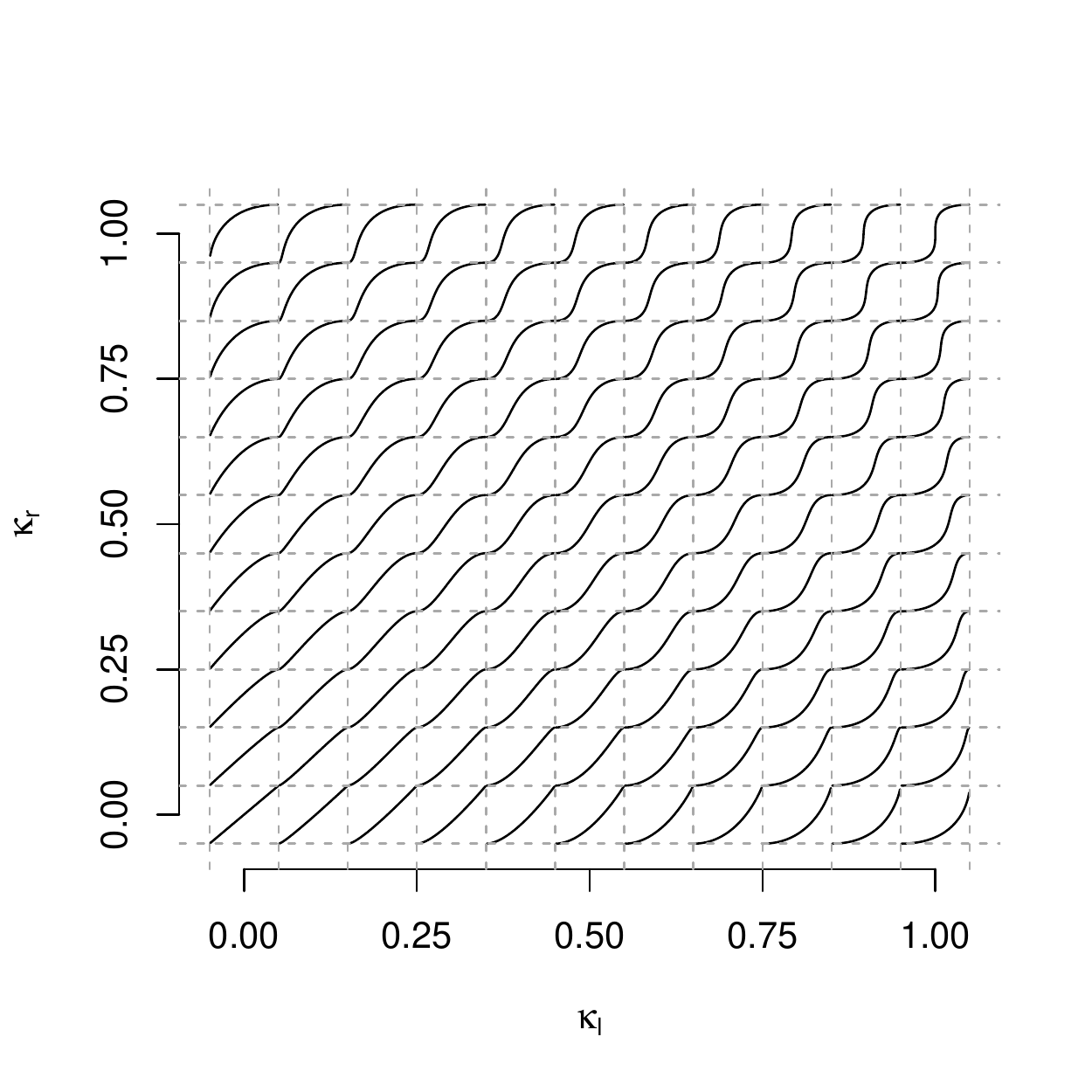}
			\caption{BMT cumulative distribution functions.}
			\label{fig:BMTcdfs}
		\end{subfigure}
		\begin{subfigure}{0.475\textwidth}
			\includegraphics[width=\textwidth]{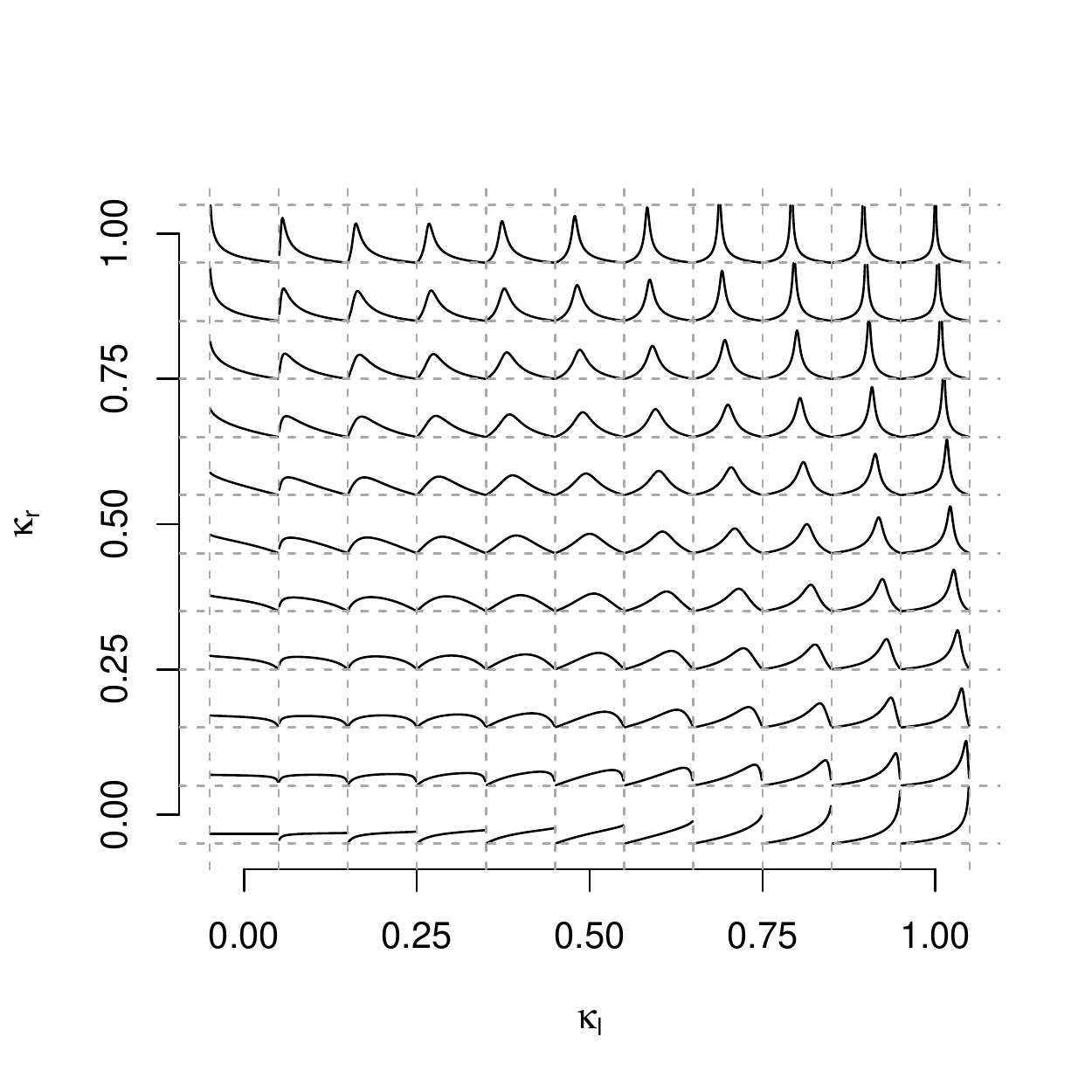}
			\caption{BMT probability density functions.}
			\label{fig:BMTpdfs}
		\end{subfigure}
		\caption{Shape plot for the BMT cumulative distribution and probability density functions.}
		\label{fig:BMTcdfspdfs}
	\end{figure}
	
	When $(\kappa_l,\kappa_r) \to (0,0)$, both tails become flat and we obtain the uniform continuous distribution. With regard to other limiting cases, we consider important to mention that:
	\begin{itemize}
		\item $\lim\limits_{t \to 0.5} \left( \lim\limits_{(\kappa_l,\kappa_r) \to (1,1)} \frac{2t(1-t)}{\left(3\kappa _{l}+3\kappa_r-2\right) t^{2} + \left( -4\kappa _{l}-2\kappa_r+2\right) t + \kappa _{l}} \right) = \lim\limits_{t \to 0.5} \frac{2t(1-t)}{4 t^2 - 4 t + 1} = \infty$,
		\item $\lim\limits_{t \to 0} \left(\lim\limits_{(\kappa_l,\kappa_r) \to (0,1)} \frac{2t(1-t)}{\left(3\kappa _{l}+3\kappa_r-2\right) t^{2} + \left( -4\kappa _{l}-2\kappa_r+2\right) t + \kappa _{l}} \right) = \lim\limits_{t \to 0} \frac{2t(1-t)}{t^2} = \infty$, and,
		\item $\lim\limits_{t \to 1} \left(\lim\limits_{(\kappa_l,\kappa_r) \to (1,0)} \frac{2t(1-t)}{\left(3\kappa _{l}+3\kappa_r-2\right) t^{2} + \left( -4\kappa _{l}-2\kappa_r+2\right) t + \kappa _{l}}\right) = \lim\limits_{t \to 1} \frac{2t(1-t)}{t^2 - 2 t + 1} = \infty$.
	\end{itemize}
	
	Moreover, since the tangent line slope of a BMT PDF is given by, 
	\[\frac{\mathrm{y}_\mathrm{f}'(t)}{\mathrm{x}'(t)} = \frac{2}{3}\frac{-\left( \kappa _{r}-\kappa _{l}\right) t^{2}-2\kappa _{l}t+\kappa _{l}}{\left( \left( 3\kappa _{l}+3\kappa _{r}-2\right) t^{2}+\left( -4\kappa _{l}-2\kappa _{r}+2\right) t+\kappa _{l}\right) ^{3}},\] 
	the BMT distribution always has one mode, which is given by,
	\[ Mode[X] = 
	\begin{cases}
		\mathrm{x}\left(0.5\right)=0.5 & \text{if } \kappa_l = \kappa_r \text{ (symmetric case)} \\
		\mathrm{x}\left(\frac{\sqrt{\kappa_{l}\kappa_r}-\kappa_{l}}{\kappa_r-\kappa_l}\right) & \text{if } \kappa_l \neq \kappa_r  \text{ (skewed case)}
	\end{cases} \]
	
	\section{Quantile function and simulation} \label{sec:quantile}
	
	The quantile function of a random variable $X$, with CDF given by $\mathrm{F}_X(x) = \mathrm{y}_\mathrm{F} \left( \mathrm{x}^{-1}(x) \right)$, is $\mathrm{F}_X^{-1}(p) = \mathrm{x} \left( \mathrm{y}_\mathrm{F}^{-1}(p) \right)$, where $p \in [0,1]$. To establish $\mathrm{y}_\mathrm{F}^{-1}$ for the BMT distribution, we need to find $t \in [0,1]$ such that,
	\begin{equation} \label{eq:BMTYroot}
	\mathrm{y}_\mathrm{F}(t) = - 2 t^{3} + 3 t^{2} = p
	\end{equation}
	
	The solution to Equation (\ref{eq:BMTYroot}) can be computed by any root-finding algorithm. However, an efficient and accurate way to get real roots of a cubic polynomial is using Francois Viete's equations \cite[Section 5.6]{press2007numerical}. Hence, for $p=0$, the root is $t=0$; for $p=1$, the root is $t=1$; and for $p \in (0,1)$, the only real root on the interval $(0,1)$ is given by,
	\begin{equation} \label{eq:BMTYinverse}
	\mathrm{y}_\mathrm{F}^{-1}(p) = \frac{1}{2}-\cos \left( \frac{\arccos \left( 2p-1\right)-2\pi}{3} \right).
	\end{equation}
	
	Therefore, the quantile function $\mathrm{F}_X^{-1}$ has a close-form expression. As a result, the median of $X$ is,
	\begin{equation}
	Median[X] = \mathrm{x} \left( \mathrm{y}_\mathrm{F}^{-1}(0.5) \right) = \frac{1}{2} - \frac{3}{8} (\kappa_r - \kappa_l),
	\end{equation}
	and the interquartile range of $X$ is,
	\begin{equation}
	IQR[X] = \mathrm{x} \left( \mathrm{y}_\mathrm{F}^{-1}(0.75) \right) - \mathrm{x} \left( \mathrm{y}_\mathrm{F}^{-1}(0.25) \right) = \frac{1}{2} - 3\left(\frac{1}{4} - \cos \frac{4}{9}\pi\right) \left(\kappa_r + \kappa_l\right).
	\end{equation}
	Also, the method of inversion can be used straightforward for sampling or simulation.
	
	\section{Moments} \label{sec:moments}
	
	Substituting Bezier control points $\mathbf{b}_0, \mathbf{b}_1, \mathbf{b}_2$, and $\mathbf{b}_3$ of a BMT distribution in the $r$-th raw moment of a Bezier distribution (Equation (\ref{eq:BezierRawMoment})), we have that the $r$-th raw moment of a BMT random variable $X$ is,
	\begin{equation} \label{eq:BMTRawMoment}
	E\left[ X^r \right] = \mu_{X}^{*(r)} = 2 \frac{r!}{r+1} \sum_{k_{1}+k_{2}+k_{3}=r} \frac{3^{k_{1}+k_{2}} \kappa _{l}^{k_{1}} \left( 1-\kappa_r\right) ^{k_{2}}}{k_{1}!k_{2}!k_{3}! \binom{3r+2}{1+k_{1}+2k_{2}+3k_{3}}}.
	\end{equation}
	Correspondingly, from the $r$-th central moment of a Bezier random variable (Equation (\ref{eq:BezierCentralMoment})), the $r$-th central moment of $X$ is,
	\begin{equation} \label{eq:BMTCentralMoment}
	E\left[ \left(X - \mu_X\right)^r \right] = \mu_{X}^{(r)} = 2 \frac{r!}{r+1} \sum_{k_{0}+k_{1}+k_{2}+k_{3}=r} \frac{3^{k_{1}+k_{2}} a_0^{k_{0}} a_1^{k_{1}} a_2^{k_{2}} a_3^{k_{3}}}{k_{0}!k_{1}!k_{2}!k_{3}! \binom{3r+2}{1+k_{1}+2k_{2}+3k_{3}}}, 
	\end{equation}
	where $a_0 = - \mu_X$, $a_1 = \kappa _{l} - \mu_X$, $a_2 = 1 - \kappa_r - \mu_X$, and $a_3 = 1 - \mu_X$. Therefore, mean, variance, Pearson's skewness, and Pearson's kurtosis can be derived from Equations (\ref{eq:BMTRawMoment}) and (\ref{eq:BMTCentralMoment}). We have that,
	\begin{equation}
	E\left[X\right] = \mu_{X}^{*(1)} = \frac{1}{2} - \frac{3}{10} (\kappa_r - \kappa_l),
	\end{equation}
	\begin{equation}
	Var\left[X\right] = \mu_{X}^{(2)} = \frac{1}{2100}\left( 36\kappa _{l}^{2}+36\kappa_r^{2}+18\kappa _{l}\kappa_r-120\kappa _{l}-120\kappa_r+175\right),
	\end{equation}
	\begin{equation}
	Skew\left[X\right] = \frac{\mu_{X}^{(3)}}{\left(\mu_{X}^{(2)}\right)^{3/2}} = \frac{27\sqrt{21}\left( \kappa_r-\kappa _{l}\right) \left( 13\kappa _{l}^{2}+13\kappa_r^{2}+4\kappa _{l}\kappa_r-65\kappa _{l}-65\kappa_r+150\right) }{11\left( 36\kappa _{l}^{2}+36\kappa_r^{2}+18\kappa _{l}\kappa_r-120\kappa _{l}-120\kappa_r+175\right) ^{3/2}},
	\end{equation}
	and,
	\begin{equation}
	Kurt\left[X\right] = \frac{\mu_{X}^{(4)}}{\left(\mu_{X}^{(2)}\right)^{2}},
	\end{equation}
	where,
	\begin{align*}
	\mu_{X}^{(4)} &= \frac{1}{10010000} \left( 6507\kappa _{l}^{4} + 6507\kappa_r^{4} + 432\kappa _{l}^{3}\kappa_r + 432\kappa _{l}\kappa_r^{3} + 13122\kappa _{l}^{2}\kappa_r^{2} \right. \\
	& \qquad{} \qquad{} \qquad{} - 43380\kappa _{l}^{3} - 43380\kappa_r^{3} - 28620\kappa _{l}^{2}\kappa_r - 28620\kappa _{l}\kappa_r^{2} + 29700\kappa _{l}\kappa_r \\
	& \qquad{} \qquad{} \qquad{} \left. + 135900\kappa _{l}^{2} + 135900\kappa_r^{2} - 150000\kappa _{l} - 150000\kappa_r + 125125\right).
	\end{align*}

	Appendix \ref{ap:descr} shows all the possible outcomes, for different values of $\kappa_l$ and $\kappa_r$, of the BMT descriptive measures obtained so far (mean, median, mode, variance, standard deviation, interquantile range, Person's skewness, and Pearson's kurtosis). 
	
	Regarding the usefulness of some moments, \citet{pearson1916philosophical} and \citet{cullen1999probabilistic}, among others, produce different planes to illustrate characteristics or scope of some distributions. Figure~\ref{fig:s2-k_diagram} presents a squared-skewness - kurtosis plane with the BMT and some common distributions represented on it. In the mentioned plane, the distribution of a random variable $X$ is represented by coordinates $(Skew\left[X\right]^2,Kurt\left[X\right])$. Since both coordinates are positive, only the first quadrant of the plane is needed. Also, distributional families could be represented by a point, a curve, or a region inside that plane. For example, all distributions belonging to the normal family have squared skewness equal to zero and kurtosis equal to three, regardless the values of its location and scale parameters. Then, normal distribution family is represented inside the plane by the point $(0,3)$. Another example could be the gamma distribution family with shape parameter $\alpha$ and rate parameter $\beta$. That family has coordinates $(4/\alpha,6/\alpha + 3)$, so the family is represented by all the points of the line $y=1.5x + 3$. The region for the BMT distribution in Figure~\ref{fig:s2-k_diagram} shows that this family has: symmetric shapes with kurtosis from $1.8$, equal to the continuous uniform distribution, to $6.78$, similar to the student's t distribution with $5.59$ degrees of freedom; shapes more skewed than the most asymmetrical skew-normal; and some shapes that the very flexible beta distribution cannot reach, given the BMT region above the line that represents the gamma family.
	
	\begin{figure}[t]
		\centering
		\includegraphics[width=0.6\textwidth]{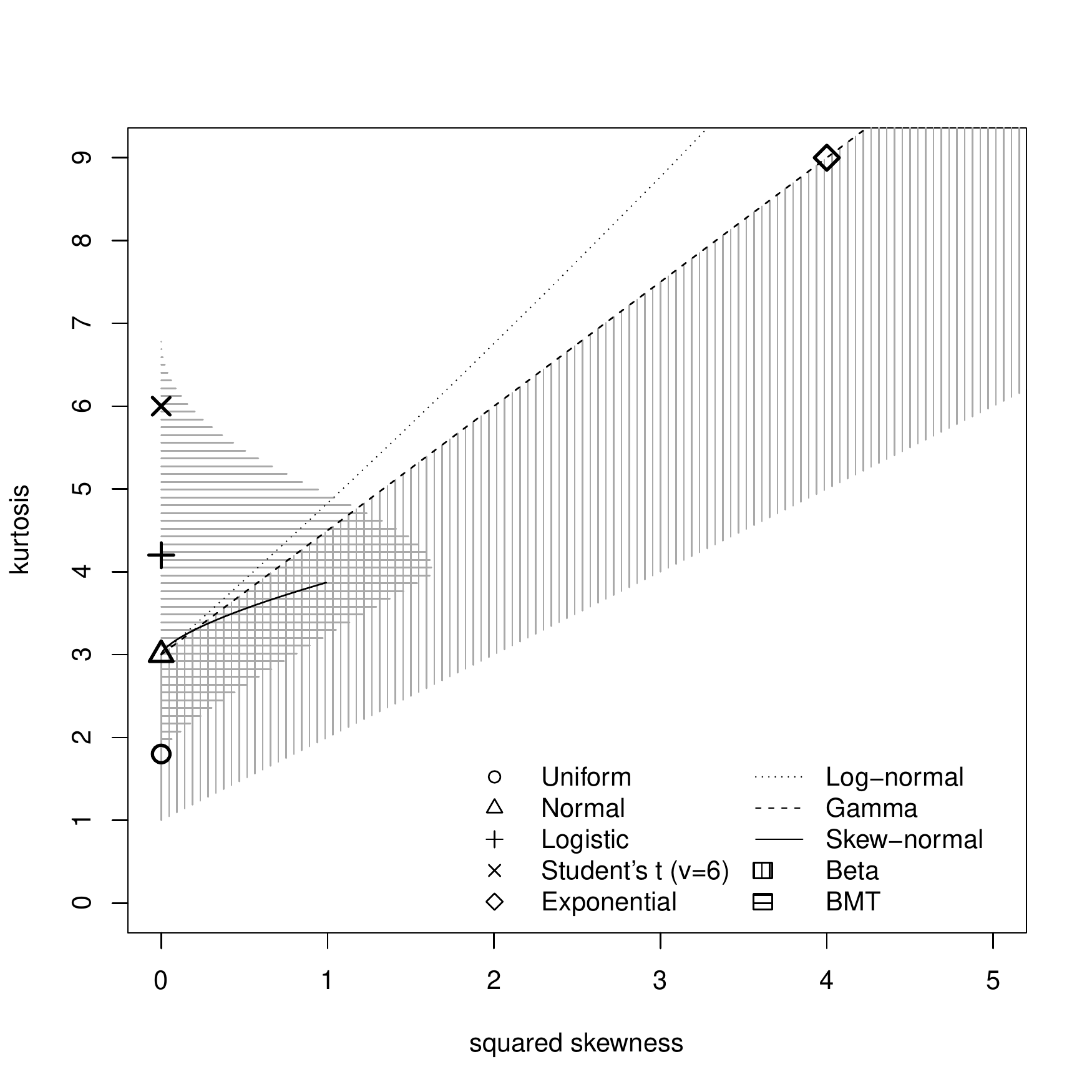}
		\caption{Squared skewness - kurtosis diagram.}
		\label{fig:s2-k_diagram}
	\end{figure}
	
	\section{Domain on $[c,d]$} \label{sec:BMT4}
	
	As any distribution with domain $[0,1]$, it is possible to alter the BMT distribution domain to $[c,d]$ by a linear transformation, introducing two further parameters $c, d \in \mathbb{R}$ ($c < d$). A random variable $Y$ is distributed BMT on $[c,d]$, denoted by $BMT\left(c,d,\kappa_l,\kappa_r\right)$, if and only if, $X = \frac{ Y - c }{ d - c } \sim BMT\left( \kappa_l,\kappa_r \right) \equiv BMT\left(0,1,\kappa_l,\kappa_r\right)$. By Proposition~\ref{prop:EBezierlinealescalar}, if $X \sim BMT\left( \kappa_l,\kappa_r \right)$, then for any $c < d$, the random variable $Y = (d-c)X + c$ belongs to the Bezier distribution family with control points $\mathbf{b}_0=\left(c,0\right)^\mathrm{T}, \mathbf{b}_1=\left((d-c)\kappa_l+c,0\right)^\mathrm{T}, \mathbf{b}_2=\left((d-c)(1-\kappa_r)+c,1\right)^\mathrm{T}, \mathbf{b}_3=\left(d,1\right)^\mathrm{T}$.
	
	If the CDF of $X$ is given by parametric equations $\mathrm{x}_X(t)$ and $\mathrm{y}_\mathrm{F_X}(t)$, the CDF of $Y$ is given by, 
	\[\mathrm{x}_Y(t) = (d-c)\mathrm{x}_X(t)+c \qquad\text{ and }\qquad \mathrm{y}_\mathrm{F_Y}(t) = \mathrm{y}_\mathrm{F_X}(t).\]
	In the same way, the pdf of $Y$ is given by, 
	\[\mathrm{x}_Y(t) = (d-c)\mathrm{x}_X(t)+c \qquad\text{ and }\qquad \mathrm{y}_\mathrm{f_Y}(t) = \frac{1}{d-c}\mathrm{y}_\mathrm{f_X}(t).\] 
	And, the closed-form expression for the quantile function of $Y$ is, 
	\[(d-c)\mathrm{x}_X\left( \mathrm{y}_\mathrm{F_X}^{-1}(p) \right) + c,\] 
	where $\mathrm{y}_\mathrm{F_X}^{-1}(0) = 0$, $\mathrm{y}_\mathrm{F_X}^{-1}(1) = 1$, and $\mathrm{y}_\mathrm{F_X}^{-1}(p) = \frac{1}{2}-\cos \left( \frac{\arccos \left( 2p-1\right)-2\pi}{3} \right)$ for $p \in (0,1)$.
	 
	The mean, median, and mode of $Y$ are those of $X$ scaled by $(d-c)$ and shifted by $c$; variance is scaled by $(d-c)^2$; interquartile range and standard deviation are scaled by $(d-c)$; and, since Pearson's skewness and kurtosis are standardized moments, they are the same for $Y$ and $X$.
	
	\section{Estimation} \label{sec:estimation}
	
	Let $\boldsymbol{\theta} = \left(\kappa_l,\kappa_r\right)^\mathrm{T}$ be the parameter vector and $\boldsymbol{\Theta} = (0,1) \times (0,1)$ the parameter space of a BMT distribution. Since the PDF of the BMT distribution does not have an explicit formula, a numerical approach is needed to obtain a maximum likelihood estimate (MLE) \citep{fisher1922mathematical},
	\[\hat{\boldsymbol{\theta}}_{mle} = \arg \max_{\boldsymbol{\theta} \in \boldsymbol{\Theta}} \left\{ \sum_{i=1}^{n} \ln \left( \mathrm{f} \left( x_{i}; \boldsymbol{\theta} \right) \right) \right\}. \]
	
	Nonetheless, a MLE might not exist and observations close to $0$, $1$, or $0.5$ and $\boldsymbol{\theta}$ in the vicinity of $(0,1)^\mathrm{T}$, $(1,0)^\mathrm{T}$, or $(1,1)^\mathrm{T}$ could give some trouble to the solving mechanism of the maximum likelihood optimization problem.
	
	Considering potential inconveniences with maximum likelihood estimation, we explore an alternative method. The maximum product of spacing estimate (MPSE) \citep{cheng1983estimating}, also called maximum spacing estimate \citep{ranneby1984maximum}, conserves some properties and surpasses some difficulties of the maximum likelihood estimation. 
	
	In general, for a density $\mathrm{f}_X(x;\boldsymbol{\theta})$ strictly positive in the interval $(x_{(0)},x_{(n+1)})$ and zero outside of it, and an ordered random sample $x_{(0)} < x_{(1)} < \dots < x_{(n)} < x_{(n+1)}$, where $x_{(0)}$ and $x_{(n+1)}$ could be known or unknown values, MPSEs are $\hat{\boldsymbol{\theta}}_{mpse} \in \boldsymbol{\Theta}$, such that they maximize the sum (or arithmetic mean) of the logarithm of spacings: $\mathrm{F}(x_{(i)};\boldsymbol{\theta}) - \mathrm{F}(x_{(i-1)};\boldsymbol{\theta})$, for $i = 1,\dots,n+1$,
	
	\[ \hat{\boldsymbol{\theta}}_{mpse} = \arg \max_{\boldsymbol{\theta} \in \boldsymbol{\Theta}} \left\{ \sum_{i=1}^{n+1} \ln \left( \mathrm{F} \left( x_{(i)}; \boldsymbol{\theta} \right) - \mathrm{F} \left( x_{(i-1)}; \boldsymbol{\theta} \right) \right) \right\}. \]
	
	\citet{cheng1983estimating} and \citet{ranneby1984maximum} propose the maximum product of spacing method by two separated ways, the first motivated on the probability integral transform and the second one on the Kullback-Leibler divergence. Both works show that MLEs and MPSEs are related and demonstrate that the maximum product of spacing method can achieve consistent, asymptotically normal, and asymptotically efficient estimators, when a MLE exists and under more general conditions.
	
	The only known downside about the maximum product of spacing method is when $x_{(i)} = x_{(i-1)}$. In that case, the respective spacing can be replaced by $\mathrm{f}(x_{(i)};\boldsymbol{\theta})$. But, if $x_{(i)} = x_{(0)}$ or $x_{(i)} = x_{(n+1)}$, the observation is standardly ignored or excluded, just like with the maximum likelihood method. 
	
	The objective function of the maximum product of spacing optimization problem is bounded, thus it always has at least a supreme. Since the CDF of a BMT distribution does not have a close-form formula, a MPSE will also have to be found numerically.
	
	To test estimation methods together with optimization algorithms, we run some simulations and check parameter recovery. We simulate $1000$ samples of size $30$, $300$, and $3000$, from a BMT distribution on $[0,1]$ with parameter vector $\boldsymbol{\theta}_1 = \left(0.5,0.5\right)^T$, $\boldsymbol{\theta}_2 = \left(0.2,0.4\right)^T$, and $\boldsymbol{\theta}_3 = \left(0.9,0.1\right)^T$. For each sample, we employ a trust-region approach to box-constrained optimization \citep{gay1984trust} for a numerical maximum likelihood and maximum product of spacing estimation. Function \verb|nlminb| of the software \verb|R| \citep{RCoreTeam2015language} was used and $(0.6,0.6)^T$ was always the initial value for $\boldsymbol{\theta}$. Following the estimation, we calculate the absolute difference between a parameter value and the obtained estimate for each simulated sample. By sample size, parameter vector, and estimation method, the mean, median, and standard deviation of the mentioned differences were computed for each set of samples (See Table \ref{tab:ParRecov} in Appendix \ref{ap:ParRecov}). Results indicate that we have successful numerical procedures for parameter estimation. Also, despite the difficulties associated to analytically solving the MLE optimization problem for the BMT distribution, a numerical solution does not have considerable inconveniences, at least for the arbitrary chosen parameter vectors along with the selected optimization method.
	
    \section{Applications} \label{sec:applications}
	
	In this section, we illustrate the usefulness and potential of the BMT distribution with the help of three real data sets.
	
	\subsection{PISA 2012}

	% Cargar los datos.
	% Identificar las columnas que corresponden a preguntas (scored items). 
	% Primer caracter P: Paper C: Computer D:?. Último caracter S: Scored.
	% Identificar las columnas que corresponden a items politómicos (partial credit).
	% Quedarse con los scored items que no son politómicos (dicotómicos).
	% Recodificar los ítems dicótomicos (1: Full credit, 0: resto, NA: N/A).
	% Contar o calcular el número de respuestas correctas por "constructo".
	% Segundo caracter M: Math R:Reading S:Science.

	Programme for International Student Assessment (PISA) of the Organisation for Economic Co-operation and Development (OECD) aims to evaluate educational systems. Every three years since 2000, PISA has been designed, applied, and studied surveys about literacy of 15-year-old school students. The questionnaires mainly evaluate performance in mathematics, science, and reading. Those tests have multiple-choice and open-ended questions setting up in real life situations, independent of schools curriculum as much as possible. The PISA 2012 assessment evaluated around $510 000$ students of $65$ countries or economies, representing approximately $28$ million individuals worldwide.
	
	In this first application, we use the answers to the PISA 2012 questionnaire 
	%\footnote{PISA 2015 database is going to be published on 6th December 2016 at 11.00 am (Central European Time).} 
	\citep{oecd2012database}. First, we take the ``Scored cognitive item response data file.''. We keep all the questions with binary response (correct and incorrect) and exclude those that could be scored with partial credit. Then, we recode the responses: $1$ for correct and $0$ for incorrect. Finally, we obtain the percentage of correct answers, i.e., the classic performance score of each student. It is important to mention that the reported scores of PISA 2012 are estimated and scaled using the Rasch model of item response theory. On the other hand, all booklets for the test could have different: number of questions, traits evaluated, and participating countries. Also, the assignation of a booklet to a student is randomized. Considering that, we choose only one arbitrary booklet, Booklet 10, and its questions of mathematics. In conclusion, the variable to be fitted or modeled by the BMT distribution is precisely the classic performance score in mathematics, using the students responses to Booklet 10 of the PISA test applied in 2012.
	
	Table~\ref{tab:PISA_descr} displays some summary statistics of the sample. The Booklet 10 was given to $35545$ students, $7.32\%$ of the evaluated people that year. The performance as a percentage naturally goes from $0\%$ to $100\%$, and, indeed, there is no reason that would impede answering correctly or wrongly all the dichotomous math questions from Booklet 10. Sample skewness and kurtosis allow us to locate the data in Figure~\ref{fig:s2-k_diagram}, slightly to the right from axis $y$ and just between the points representing normal and uniform distributions. 
	
	\renewcommand\baselinestretch{1}
	\begin{table}[t]
		\centering
		\caption{Descriptive statistics of PISA 2012 data.}
		\begin{tabular}{cccccccc} \hline
			n & min & max & median & mean & sd & skewness & kurtosis \\ \hline
			$35545$ & $0.00$ & $1.00$ & $0.41$ & $0.44$ & $0.21$ & $0.27$ & $2.23$ \\ \hline
		\end{tabular}
		\label{tab:PISA_descr} \bigskip \bigskip
		
		\caption{MLEs and MPSEs of the distribution shape parameters for PISA 2012 data, and objective functions evaluated at $\hat{\boldsymbol{\theta}}$.}
		\begin{tabular}{llccc} \hline
			 & Method & $\hat{\boldsymbol{\theta}}$ & logLik & Sum log spacings \\ \hline
			Beta & MLE & $\left(1.9764, 2.4744\right)$ & $6040.0500$ & $5925.7131$ \\
			 $\boldsymbol{\theta} = (\alpha,\beta)$ & MPSE & $\left(1.9759, 2.4737\right)$ & $6040.0493$ & $5925.7138$ \\  \hline
			Kumaraswamy & MLE & $\left(1.7532, 2.5818\right)$ & $5913.3006$ & $5799.0648$ \\ 
			 $\boldsymbol{\theta} = (a,b)$ & MPSE & $\left(1.7529, 2.5810\right)$ & $5913.2999$ & $5799.0655$ \\  \hline
			BMT & MLE & $\left(0.2852, 0.4871\right)$ & $6138.6886^*$ & $6024.7309$ \\
			 $\boldsymbol{\theta} = (\kappa_l,\kappa_r)$ & MPSE & $\left(0.2852, 0.4871\right)$ & $6138.6885$ & $6024.7310^*$ \\ \hline
			\multicolumn{5}{l}{$^*$Highest value for the objective function.}
		\end{tabular}
		\label{tab:PISA_Es}
	\end{table}
	
	The above suggests that it is appropriate to fit the data with distributions like the beta, Kumaraswamy, or BMT (on $[0,1]$). By maximum likelihood and maximum product of spacing, we obtain estimates of the shape parameters for each of those distributions. Table~\ref{tab:PISA_Es} shows attained MLEs and MPSEs for the beta, Kumaraswamy, and BMT distributions, and, in each case, objective functions associated to both estimation methods were evaluated at the attained estimate $\hat{\boldsymbol{\theta}}$. For MLEs, the objective function is the natural logarithm of the likelihood, and, for MPSEs, the objective function is the sum of the natural logarithm of spacings.
	
	From Table~\ref{tab:PISA_Es}, we can see that maximum likelihood and maximum product of spacing have almost the same results. The reason of this might be a large enough sample size and/or many equal observations. In addition, BMT achieves the highest values for the respective objective functions among the selected distributions. Since all the models have the same number of parameters, Akaike information criterion (AIC) and Bayesian information criterion (BIC) will also indicate that the BMT provides to some extent a better fit than the other two distributions. With regard to the values of $\hat{\boldsymbol{\theta}}$ for the BMT distribution, we can say that we establish an estimated BMT curvature degree of $28.52\%$ for the left tail $(\kappa_l)$ and of $48.71\%$ for the right tail $(\kappa_r)$. Right tail is steeper than the left one and that implies a right-skewed estimated distribution, with an asymmetry of $20.19$ BMT percentage points if we use $\kappa_r - \kappa_l$ as an asymmetry indicator.
			
	\subsection{Food Expenditure}
			
	The data of our second application correspond to the proportion of income spent on food, used for a beta regression model application \citep{ferrari2004beta}. The source of this data \cite[Table 15.4]{griffiths1993learning} has the income, food expenditure, and number of people in a sample of 38 households from a large U.S. city. \citet{ferrari2004beta} use the mentioned proportion as response of their proposed regression.
			
	Food expenditure, as a proportion of the income, theoretically goes from $0$ to $1$. We consider summary statistics of the variable (Table~\ref{tab:Food_descr}). According to skewness close to one and kurtosis slightly above four, the data gets inside beta and BMT regions in Figure~\ref{fig:s2-k_diagram}. Then, it is reasonable to follow the same procedure as with the previous application. Table~\ref{tab:Food_Es2} shows that the BMT distribution has the highest sum of log spacings, and is the only one with similar estimates for both estimation methods. On the other hand, the beta distribution has the highest log likelihood, and, as a result, the lowest AIC and BIC among the selected distributions.
			
	\begin{table}[t]
		\centering
		\caption{Descriptive statistics of food expenditure proportion.}
		\begin{tabular}{cccccccc} \hline
			n & min & max & median & mean & sd & skewness & kurtosis \\ \hline
			$38$ & $0.11$ & $0.56$ & $0.26$ & $0.29$ & $0.10$ & $0.98$ & $4.16$ \\ \hline
		\end{tabular}
		\label{tab:Food_descr} \bigskip \bigskip
		
		\caption{MLEs and MPSEs of the distribution shape parameters for food expenditure proportion, and objective functions evaluated at $\hat{\boldsymbol{\theta}}$.}
		\begin{tabular}{llccc} \hline
			& Method & $\hat{\boldsymbol{\theta}}$ & logLik & Sum log spacings \\ \hline
			Beta & MLE & $\left(6.0716, 14.8221\right)$ & $35.3464^*$ & $-162.3005$ \\ 
			$\boldsymbol{\theta} = (\alpha,\beta)$ & MPSE & $\left(5.0982, 12.3509\right)$ & $35.0435$ & $-161.9838$ \\  \hline
			Kumaraswamy & MLE & $\left(2.9546, 26.9654\right)$ & $33.4891$ & $-163.6891$ \\ 
			$\boldsymbol{\theta} = (a,b)$ & MPSE & $\left(2.7211, 20.1626\right)$ & $33.2275$ & $-163.4152$ \\  \hline
			BMT & MLE & $\left(0.4304, 1.0000\right)$ & $33.2552$ & $-161.8100$ \\
			$\boldsymbol{\theta} = (\kappa_l,\kappa_r)$ & MPSE & $\left(0.4281, 1.0000\right)$ & $33.2523$ & $-161.8071^*$ \\ \hline
			\multicolumn{5}{l}{$^*$Highest value for the objective function.}
		\end{tabular}
		\label{tab:Food_Es2} \bigskip \bigskip
		
		\caption{MLEs and MPSEs of the distribution domain and shape parameters for food expenditure proportion, and objective functions evaluated at $\hat{\boldsymbol{\theta}}$.}
		\begin{tabular}{llccc} \hline
			& Method & $\hat{\boldsymbol{\theta}}$ & logLik & Sum log spacings \\ \hline
			Beta & MLE & $\left(0.04, 0.80 \times 10^6, 6.31, 20.37 \times 10^6\right)$ & $36.2398$ & $-161.5148$ \\
			$\boldsymbol{\theta} = (c,d,\alpha,\beta)$ & MPSE & $\left(0.01, 1.05 \times 10^6, 7.14, 26.09 \times 10^6\right)$ & $35.8634$ & $-161.1030$ \\  \hline
			Kumaraswamy & MLE & $\left(0.09,    157.27,      2.06, 0.73 \times 10^6\right)$ & $35.6851$ & $-162.6788$ \\
			$\boldsymbol{\theta} = (c,d,a,b)$ & MPSE & $\left(0.07,      646.33,        2.09, 13.87 \times 10^6\right)$ & $35.1539$ & $-162.0405$ \\ 	 \hline
			BMT & MLE & $\left(0.08, 0.65, 0.43, 0.86\right)$ & $37.1966^*$ & $-160.7914$ \\
			$\boldsymbol{\theta} = (c,d,\kappa_l,\kappa_r)$ & MPSE & $\left(0.03, 0.73, 0.51, 0.94\right)$ & $36.4574$ & $-159.8543^*$ \\ \hline
			\multicolumn{5}{l}{$^*$Highest value for the objective function.}
		\end{tabular}
		\label{tab:Food_Es4}
	\end{table}
						
	To be more precise, food expenditure proportion of zero or one does not seem to have sense in practice. For every household, some of the income should go to food and also to something else than food. Indeed, sample minimum and maximum say that food expenditure percentage goes from $11\%$ to $56\%$ for the 38 households. Therefore, we believe that it is more suitable to use a distribution on $[c,d]$ than on $[0,1]$. In addition, estimates for the population minimum and maximum can be considered of special interest.
						
	We obtain Table~\ref{tab:Food_Es4}, extending the same worked distributions and methods to the inclusion of parameters $c$ and $d$. Between those extended distributions, the BMT achieves the highest values for the objective function of both estimation methods and presents reasonable estimates for minimum and maximum population proportions of income spent on food. On the other hand, estimated $d$ for beta and Kumaraswamy distributions are not valid proportions. Table~\ref{tab:Food_Es4} also gives us hints about differences between maximum likelihood and maximum product of spacing for distributions with four (two domain and two shape) parameters and a small sample.
				
	The plots of the fitted densities by maximum likelihood for the beta with two parameters, the beta with four parameters, and the BMT with four parameters are shown in Figure~\ref{fig:Food}. They illustrate that the BMT of four parameters provides a better fit than the other two distributions. The beta on $[0,1]$ does no achieve the observed steepness in the histogram, and, although the beta on $[c,d]$ is a little more steeper, it does not have a domain within $[0,1]$. Actually, the beta and Kuramaswamy distributions can have tails tightly attached to the x-axis, and therefore, estimates of domain parameters can go very far from minimum and maximum of the sample to achieve a better fit. From the fitted BMT on $[c,d]$, it is estimated that the population proportions of income spent on food go from $7.66\%$ to $64.93\%$, with a BMT curvature degree of $43.02\%$ for the left tail and $86.37\%$ for the right one. BMT curvature degree difference between tails imply a right-skewed distribution with an asymmetry of $\kappa_r - \kappa_l = 43.35$ BMT percentage points.

	\begin{figure}[t]
	\centering
	\includegraphics[width=0.9\textwidth]{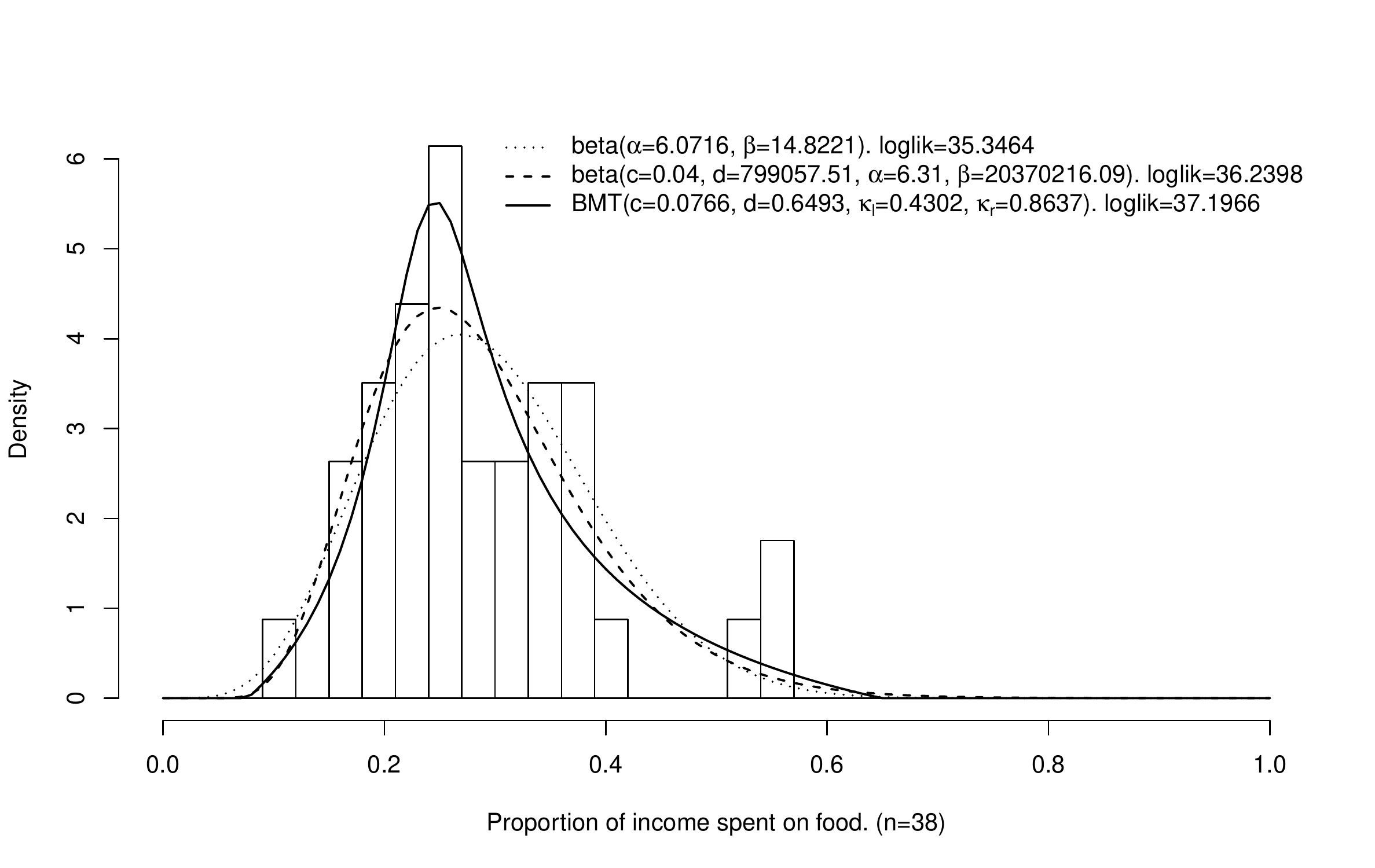}
	\caption{Maximum-likelihood fitted beta on $[0,1]$, beta on $[c,d]$, and BMT on $[c,d]$ densities for the proportion of income spent on food, in a sample of 38 households from a large U.S. city.\cite[Table 15.4]{griffiths1993learning}}
	\label{fig:Food}
    \end{figure}

	\subsection{Height of sons}
						
	For our third application, we wanted to explore the effectiveness of the BMT as a distribution supported on a bounded interval completely outside $[0,1]$. 
						
	We take the famous dataset on relationship between heights of fathers and their sons by \citet{pearson1903laws}. This data set has $1078$ observations and two variables: father's and son's height. The original data were reported to the nearest inch. Later, a small amount of random uniform noise was added to render it continuous \cite[dataframe father.son]{verzani2015usingr}, and for this example, we convert it to centimeters. With this particular application, we want to focus on estimating the stature of tallest and shortest son of the population from which the sample was taken.
						
	\begin{table}[t]
		\centering
		\caption{Descriptive statistics of son's adult height.}
		\begin{tabular}{cccccccc} \hline
			n & min & max & median & mean & sd & skewness & kurtosis \\ \hline
			$1078$ & $148.61$ & $199.05$ & $174.28$ & $174.46$ & $7.15$ & $-0.04$ & $3.54$ \\ \hline
		\end{tabular}
		\label{tab:Height_descr} \bigskip \bigskip
		
		\caption{MLEs and MPSEs of the distribution domain and shape parameters for son's adult height, and objective functions evaluated at $\hat{\boldsymbol{\theta}}$.}
		\begin{tabular}{llccc} \hline
			& Method & $\hat{\boldsymbol{\theta}}$ & logLik & Sum log spacings \\ \hline
			Beta & MLE & $(-3.73, 0.65, 32.67, 4.02) \times 10^3$ & $-3649.4808^*$ & $-8154.9793$ \\
			$\boldsymbol{\theta} = (c,d,\alpha,\beta)$ & MPSE & $(-9.99, 0.69, 96.00, 4.85) \times 10^3$ & $-3649.5255$ & $-8154.9287^*$ \\ \hline
			Kumaraswamy & MLE & $(145.50, 1024.71, 4.43, 2.50 \times 10^6)$ & $-3662.2731$ & $-8170.0036$ \\ 
			$\boldsymbol{\theta} = (c,d,a,b)$ & MPSE & $(144.83, 1065.44, 4.50, 3.47 \times 10^6)$ & $-3662.4585$ & $-8169.7991$ \\  \hline
			BMT & MLE & $(148.26, 199.61,   0.72,   0.70 )$ & $-3663.5402$ & $-8172.2421$ \\ 
			$\boldsymbol{\theta} = (c,d,\kappa_l,\kappa_r)$ & MPSE & $(147.89, 200.04,   0.73,   0.71)$ & $-3664.2763$ & $-8171.3093$ \\ \hline
			\multicolumn{5}{l}{$^*$Highest value for the objective function.} \\
		\end{tabular}
		\label{tab:Height_Es}
	\end{table}
								
	Sample skewness and kurtosis point out an approximately symmetric distribution between the normal and logistic distributions (See Table~\ref{tab:Height_descr} along with Figure~\ref{fig:s2-k_diagram}). Table~\ref{tab:Height_Es} shows the results of a maximum likelihood and maximum product of spacing estimation for the four parameter beta, Kumaraswamy, and BMT distributions. BMT distribution has the lowest values for the objective functions to maximize, but it is the only one with plausible values for parameters $c$ and $d$. Therefore, BMT is the only one of those distributions useful for our interest (stature of the tallest and the shortest son).
								
	Based on the BMT distribution and the maximum likelihood method, we estimate that the population height of tallest and smallest son are $1.48$ and $2.00$ meters, respectively. We also have a curvature degree of $72\%$ for the left tail, a curvature degree of $70\%$ for the right tail, and a very small skewness to the left with an asymmetry of $\kappa_r - \kappa_l = -1.88$ BMT percentage points.
								
	On the other hand, the normal and logistic distributions cannot give us estimates for tallest and shortest height. However, if we use them, the log likelihood function of normal and logistic distributions evaluated at their MLEs are $-3649.5634$ and $-3645.7559$, respectively. Also, $1.48$ and $2.00$ meters are the $0.01\%$ and $99.98\%$ percentiles of the estimated normal distribution, and those same heights are the $0.13\%$ and $99.83\%$ percentiles of the estimated logistic distribution. Considering what happens with the normal and logistic distributions, a truncated (skew) logistic distribution do not neglect our mentioned interest and should have a better fit than beta, Kumaraswamy, BMT, normal and logistic distributions.
								
	\section{Conclusion and comments} \label{sec:conclusions}
								
	We proposed a new double-bounded continuous distribution called BMT. As far as we know, this is one of few distribution families given by parametric equations with a small number of parameters. BMT distribution can be seen as a particular case of our more general definition of the Bezier distribution. As a result, the BMT is a quite flexible unimodal distribution with two shape parameters, on $(0,1)$, that can be interpreted as the curvature degree of each of its tails. 
	
	We also studied some general properties of the BMT distribution. Closed-form expressions for quantile function and some descriptive measures were derived. Given the formula of the BMT quantile function, an easy and fast way of sampling a BMT random variable is possible. Mean and median of a BMT distribution are linear transformations of parameters difference $(\kappa_r - \kappa_r)$, and that difference can be seen as an indicator of asymmetry. An overview of a comparison between beta, Kumaraswamy, and BMT distributions is given as a checklist in Table~\ref{tab:comparison}.
	
	\begin{table}[ht]
		\centering
		\caption{Comparison of BMT, beta, and Kumaraswamy distributions.}
		\begin{tabular}{l|ccc}
			& Beta & Kumaraswamy & BMT \\ 
			\hline Closed-form expression for cdf and pdf? &  & \checkmark &  \\ 
			\hline Closed-form expression for quantile function? &  & \checkmark & \checkmark \\ 
			\hline Closed-form expression for mean, variance, & \multirow{2}{*}{\checkmark} & \multirow{2}{*}{ } & \multirow{2}{*}{\checkmark} \\ 
			skewness, and kurtosis? \\
			\hline Symmetric shapes? & \checkmark &  & \checkmark \\ 
			\hline Different shapes aside from unimodal & \multirow{2}{*}{\checkmark} & \multirow{2}{*}{\checkmark} & \multirow{2}{*}{ } \\ 
			(U, J, reverse J shapes)? \\
			\hline 
		\end{tabular}
		\label{tab:comparison}
	\end{table}
									
	In addition to the properties mentioned, simulations and three applications show our distribution functionality. Maximum likelihood and maximum product spacing methods, in conjunction with the box-constrained optimization proposed by \citet{gay1984trust} (implemented in \verb|R|'s function \verb|nlminb|), performed well and converged for all cases, with diverse sample sizes and with two or four parameters. 
	
	The BMT distribution clearly stands out for its suitability when it comes to estimate plausible domain parameters. Not only that, but it could be useful to handle unimodal data otherwise questionably assumed on the whole real line or on a semi-infinite interval. Applications showed that four parameter beta and Kumaraswamy, with their possibility of very light tails, lead to estimates for domain parameters very far from sample minimum and maximum; while that does not happen with the BMT distribution. Equally noteworthy, Figure~\ref{fig:s2-k_diagram} shows that the BMT distribution can handle data that the beta distribution do not, given their possible values of (population) skewness and kurtosis. All the above ensures that the BMT distribution is a genuine alternative to existent continuous univariate distributions supported on a bounded interval.
	
	With regard to the computational aspect, we note that the optimization algorithms from \verb|optimx| \citep{nash2011unifying} perform very nicely solving a two parameter estimation problem for the beta, the Kuramaswamy, or the BMT distribution. First and second order Kuhn-Karush-Tucker (KKT) optimality conditions were numerically satisfied for simulations and applications with two unknown parameters. 
	
	On the contrary, optimization with the four parameter beta or Kumaraswamy distributions do not work well. Additional tests showed that optimization algorithms of a four parameter estimation problem for beta and Kumaraswamy distributions are very dependent of the starting point; the second order KKT optimality condition is not met or cannot be checked; parameters at different scales are problematic; and the worst of all, two very distant estimates can lead to very close values of the objective functions. To illustrate, for the log likelihood function $(\ell)$ of the second application we have that, 
	\[\ell(\boldsymbol{\theta}) = \ell\left(0.04, 7.99 \times 10^5, 6.31, 2.04 \times 10^7\right) = 36.2397826,\] 
	and 
	\[\ell(\boldsymbol{\theta}) = \ell\left(0.04, 7.99 \times 10^7, 6.31, 2.04 \times 10^9\right) = 36.2397827.\]
	By comparison, to the optimization algorithms, the four parameter estimation problem for the BMT distribution seems to be as well-behaved as the two parameter problem. Even if domain parameters are on a very different scale from the BMT shape parameter, only for the BMT distribution, a linear transformation of the data solves any possible issue with that difference of scales. In conclusion, to model a variable with unknown domain, we strongly recommend using the four parameter BMT distribution over the beta or Kumaraswamy distributions. 
	
	As part of first author PhD thesis, we already worked on useful alternative parametrizations and estimation methods for the BMT distribution. Regression using the BMT distribution seems straightforward, at least numerically, following the ideas of \citet{ferrari2004beta,mitnik2013kumaraswamy,cepeda2014beta,klein2015bayesian}. Since the BMT was motivated by our vision about the needs of the item response theory (IRT), we have high expectations for IRT models using the BMT distribution. Indeed, we intend to compare a proposed BMT IRT model with the skew-normal IRT model worked by \citet{bazan2006skew,azevedo2011bayesian}. 
	
	Future research is open to new mathematical properties, extensions, and applications for the BMT distribution. Likewise, comparative analysis with truncated distributions could be important and informative. 
								
	\section*{Acknowledgement(s)}
								
	The authors are thankful to the referees and editors for the useful comments.
								
	\section*{Funding}
								
	This work is a result of the first author doctoral thesis. Each semester of first author doctoral studies, an academic merit-based scholarship was granted. Also, this work was partially supported by Colciencias [grant number 0039-2013] and Universidad Nacional de Colombia [grant number DIB-2016-36008].
								
	\section*{Notes}
								
	The data processing, parameter estimation, and all the numerical calculations required for this work were performed using \verb|R| \citep{RCoreTeam2015language} and an \verb|R| package developed by the first author called \verb|BMT|, which can be found at \url{http://CRAN.R-project.org/package=BMT}. 
	In addition, some functions of the following contributed packages were used: \verb|dplyr| \citep{wickham2015dplyr}, \verb|e1071| \citep{meyer2015e1071}, \verb|fields| \citep{nychka2016fields}, \verb|fitdistrplus| \citep{delignette-muller2015fitdistrplus}, \verb|optimx| \citep{nash2011unifying}, and \verb|partitions| \citep{hankin2006additive}.
								
	\bibliographystyle{Chicago}

\newpage
	\appendix 
								
	\linespread{1.5}
	
	\section{Condition (ii) of the Definition \ref{def:BezierDist2}} \label{ap:cond2BezierDef}
								
	\begin{proposition}
		The Bezier curve associated to $\mathrm{F}_X$, given by Equations (\ref{eq:EBezier}), is a non-decreasing function, if and only if, 
		\begin{equation}\label{eq:condition2}
			\sum_{i=0}^{n-1} \left( b_{i+1,1} - b_{i,1} \right) \mathrm{B}_i^{n-1}(t) \geq 0 \text{, and, } \sum_{i=0}^{n-1} \left( b_{i+1,2} - b_{i,2} \right) \mathrm{B}_i^{n-1}(t)  \geq 0,
		\end{equation}
		for all $t \in [0,1]$.
	\end{proposition}
								
	\begin{proof}
		First, $\mathrm{F}_X$ is a non-decreasing function, if and only if, the tangent line slope of its curve is always greater or equal to zero (or $\pm \infty$ for a vertical tangent at a set of measure zero). Second, the tangent line slope of a curve given by parametric equations $\mathrm{b}_1^n(t)$ and $\mathrm{b}_2^n(t)$ is $\frac{\frac{d}{dt}\mathrm{b}_2^{n}(t)}{\frac{d}{dt}\mathrm{b}_1^n(t)}$, with vertical tangents at values of $t$ for which $\frac{d}{dt}\mathrm{b}_1^{n}(t) = 0$, provided $\frac{d}{dt}\mathrm{b}_2^{n}(t) \neq 0$. Third, since $\frac{d}{dt}\mathrm{b}_1^n(t)$ and $\frac{d}{dt}\mathrm{b}_2^n(t)$ are polynomials, for any $t^*$ such that $\frac{d}{dt}\mathrm{b}_1^n(t^*) = 0$ and $\frac{d}{dt}\mathrm{b}_2^n(t^*) = 0$, $t^*$ is a root of both, and $(t-t^*)$ can be factorized and simplified from numerator and denominator of $\frac{\frac{d}{dt}\mathrm{b}_2^{n}(t)}{\frac{d}{dt}\mathrm{b}_1^n(t)}$. Fourth, from the $r$-th derivative of a Bezier curve (Equation (\ref{eq:dBezier})), we have that $\frac{d}{dt}\mathrm{b}_1^n(t)=n \sum_{i=0}^{n-1} \left( b_{i+1,1}-b_{i,1} \right) \mathrm{B}_i^{n-1}(t)$ and $\frac{d}{dt}\mathrm{b}_2^{n}(t)=n \sum_{i=0}^{n-1} \left( b_{i+1,2}-b_{i,2} \right) \mathrm{B}_i^{n-1}(t)$. Therefore, including horizontal tangent lines ($\frac{d}{dt}\mathrm{b}_2^{n}(t) = 0$, provided $\frac{d}{dt}\mathrm{b}_1^{n}(t) \neq 0$) at a set of measure zero.
									
		$(\Leftarrow)$ If we have (\ref{eq:condition2}), then for all $t \in [0,1]$, $\frac{\frac{d}{dt}\mathrm{b}_2^{n}(t)}{\frac{d}{dt}\mathrm{b}_1^n(t)} \geq 0$. Which in turn implies that $\mathrm{F}_X$ is a non-decreasing function.
									
		$(\Rightarrow)$ If $\mathrm{F}_X$ is a non-decreasing function, then $\frac{\frac{d}{dt}\mathrm{b}_2^{n}(t)}{\frac{d}{dt}\mathrm{b}_1^n(t)} \geq 0$. And, to guarantee that inequality we have that:
									
		If $\frac{d}{dt}\mathrm{b}_1^n(t) \geq 0$ for all $t \in [0,1]$, then we must have $\frac{d}{dt}\mathrm{b}_2^n(t) \geq 0$ for all $t \in [0,1]$, and vice versa, meaning that (\ref{eq:condition2}) is fulfilled for all $t \in [0,1]$.
									
		If $\frac{d}{dt}\mathrm{b}_1^n(t) \leq 0$ for all $t \in [0,1]$, then we must have $\frac{d}{dt}\mathrm{b}_2^n(t) \leq 0$ for all $t \in [0,1]$, and vice versa, meaning that the Bezier points are indexed in ``inverse orientation.'' (as $t$ increases, the curve emerges or is graphed from right to left). The Bezier curve with control points $\mathbf{b}_n, \mathbf{b}_{n-1}, \dots, \mathbf{b}_0$ produces the same curve and fulfills (\ref{eq:condition2}) for all $t \in [0,1]$.
									
		If $\frac{d}{dt}\mathrm{b}_1^n(t) < 0$ or $\frac{d}{dt}\mathrm{b}_2^n(t) < 0$ only for $t \in \mathcal{A} \subsetneqq [0,1]$. Then both polynomials have to be simultaneously negative only in $\mathcal{A}$ (and non negative in $\mathcal{A}^C$), implying that they must have the same roots. And, if they have the same roots, then polynomials are multiples of each other, the curve is a line, Bezier control points are collinear, and we can rearrange them in such way that they fulfill (\ref{eq:condition2}) for all $t \in [0,1]$. 
	\end{proof}
								
\newpage
	\section{Raw moments of the Bezier distribution}\label{ap:rawMBezier}

	\begin{proposition}
		The $r$-th raw moment of a Bezier random variable $X$ is given by,
		\begin{equation} 
		\mu_{\;X}^{*(r)} = \frac{r!}{r+1} \sum_{k_0+\dots+k_n=r} \sum_{j=0}^{n-1} \frac{\left( \prod_{i=0}^n \binom{n}{i}^{k_i} b_{i,1}^{k_i} \right) \binom{n-1}{j} \left( b_{j+1,2} - b_{j,2} \right) }{k_0! \dots k_n! \, \binom{(r+1)n-1}{j+\sum_{i=0}^n ik_i}}.
		\end{equation}
	\end{proposition}				

	\begin{proof}
	The $r$-th raw moment of a random variable $X$ is,				
	\begin{align*}
		\mu_{\;X}^{*(r)} &= E\left[ X^r \right] = \int x^r \, \mathrm{F}_X(x) \, dx \\
		\intertext{considering that $\mathrm{F}_X(x) = \mathrm{y}_\mathrm{F} \left( \mathrm{x}^{-1}(x) \right)$, then,}
		&= \int \left( \mathrm{x}(t) \right)^r  \mathrm{y}'_\mathrm{F}(t)  \, dt,  \\
		\intertext{from the definition of a Bezier distribution (Definition \ref{def:BezierDist2}), the equation of a Bezier curve (\ref{eq:Bezier}), and the equation of a derivative of a Bezier curve (\ref{eq:dBezier}), we have that,}
		 &= \int_0^1 \left( \sum_{i=0}^{n} b_{i,1}  \mathrm{B}_i^n(t) \right)^r \left( n \sum_{j=0}^{n-1} \left( b_{j+1,2} - b_{j,2} \right) \mathrm{B}_{j}^{n-1}(t) \right)  \, dt,  \\
		\intertext{using the multinomial theorem, and accordingly, taking quantities of the form $x^0$ equal to $1$, even when $x$ equals zero,}
		 &= \int_0^1 \left( \sum_{k_0+\dots+k_n=r} \frac{r!}{k_0! \dots k_n!} \prod_{i=0}^n b_{i,1}^{k_i} \left(\mathrm{B}_i^n(t)\right)^{k_i} \right) \left( n \sum_{j=0}^{n-1} \left( b_{j+1,2} - b_{j,2} \right) \mathrm{B}_{j}^{n-1}(t) \right)  \, dt,  \\
		\intertext{rearranging,} 
		 &= n(r!) \sum_{k_0+\dots+k_n=r} \sum_{j=0}^{n-1} \left(\frac{\prod_{i=0}^n b_{i,1}^{k_i} \left( b_{j+1,2} - b_{j,2} \right)}{k_0! \dots k_n!} \int_0^1 \prod_{i=0}^n \left(\mathrm{B}_i^n(t)\right)^{k_i}  \mathrm{B}_{j}^{n-1}(t)  \, dt \right),  \\
		\intertext{since $\mathrm{B}_i^n(t) \mathrm{B}_j^m(t)= \frac{\binom{n}{i}\binom{m}{j}}{\binom{n+m}{i+j}} \mathrm{B}_{i+j}^{n+m}(t)$ \cite[Section 6.10]{farin2002curves}, then,}
		 &= n(r!) \sum_{k_0+\dots+k_n=r} \sum_{j=0}^{n-1} \left(\frac{\prod_{i=0}^n b_{i,1}^{k_i} \left( b_{j+1,2} - b_{j,2} \right)}{k_0! \dots k_n!} \frac{\binom{n-1}{j} \prod_{i=0}^n \binom{n}{i}^{k_i}}{\binom{(r+1)n-1}{j+\sum_{i=0}^n ik_i}} \int_0^1 \mathrm{B}_{j+\sum_{i=0}^n ik_i}^{(r+1)n-1}(t)  \, dt \right), \\
		\intertext{and, taking into account that $\int_{0}^{1} \mathrm{B}_i^n(t) \, dt = \frac{1}{n+1}$ \cite[Section 6.10]{farin2002curves}, we obtain that,}
		 &= \frac{r!}{r+1} \sum_{k_0+\dots+k_n=r} \sum_{j=0}^{n-1} \frac{\left( \prod_{i=0}^n \binom{n}{i}^{k_i} b_{i,1}^{k_i} \right) \binom{n-1}{j} \left( b_{j+1,2} - b_{j,2} \right) }{k_0! \dots k_n!  \binom{(r+1)n-1}{j+\sum_{i=0}^n ik_i}}.
	\end{align*}
	\end{proof}

\newpage

	\section{BMT distribution descriptive measures}\label{ap:descr}

\begin{figure}[ht]
	\centering
	\begin{subfigure}{0.32\textwidth}
		\includegraphics[width=\textwidth]{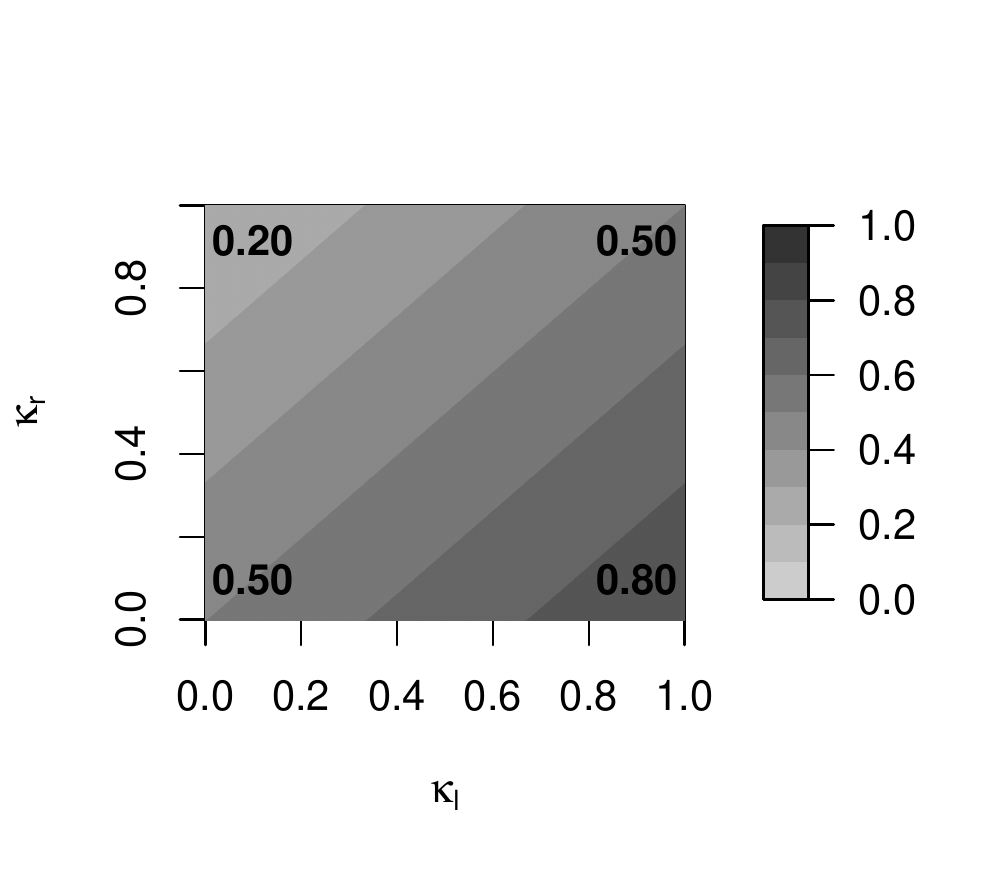}
		\caption{Mean.}
		\label{fig:BMTmean}
	\end{subfigure}
	\begin{subfigure}{0.32\textwidth}
		\includegraphics[width=\textwidth]{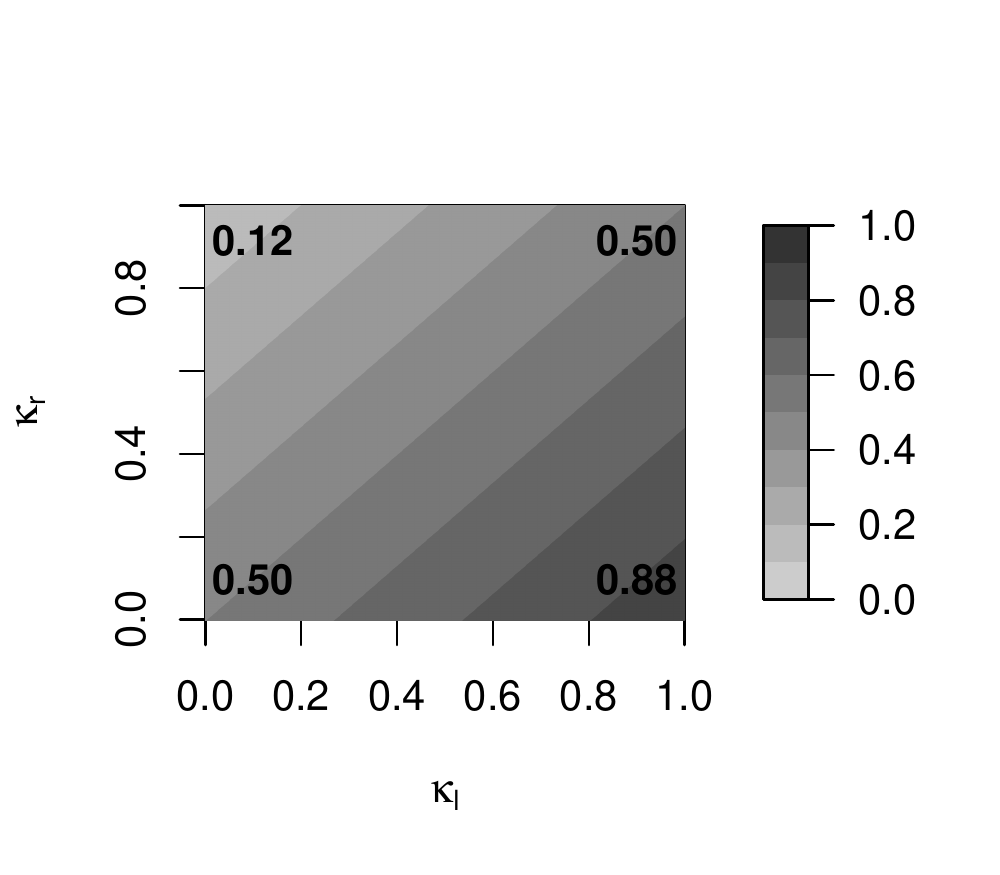}
		\caption{Median.}
		\label{fig:BMTmedian}
	\end{subfigure}
	\begin{subfigure}{0.32\textwidth}
		\includegraphics[width=\textwidth]{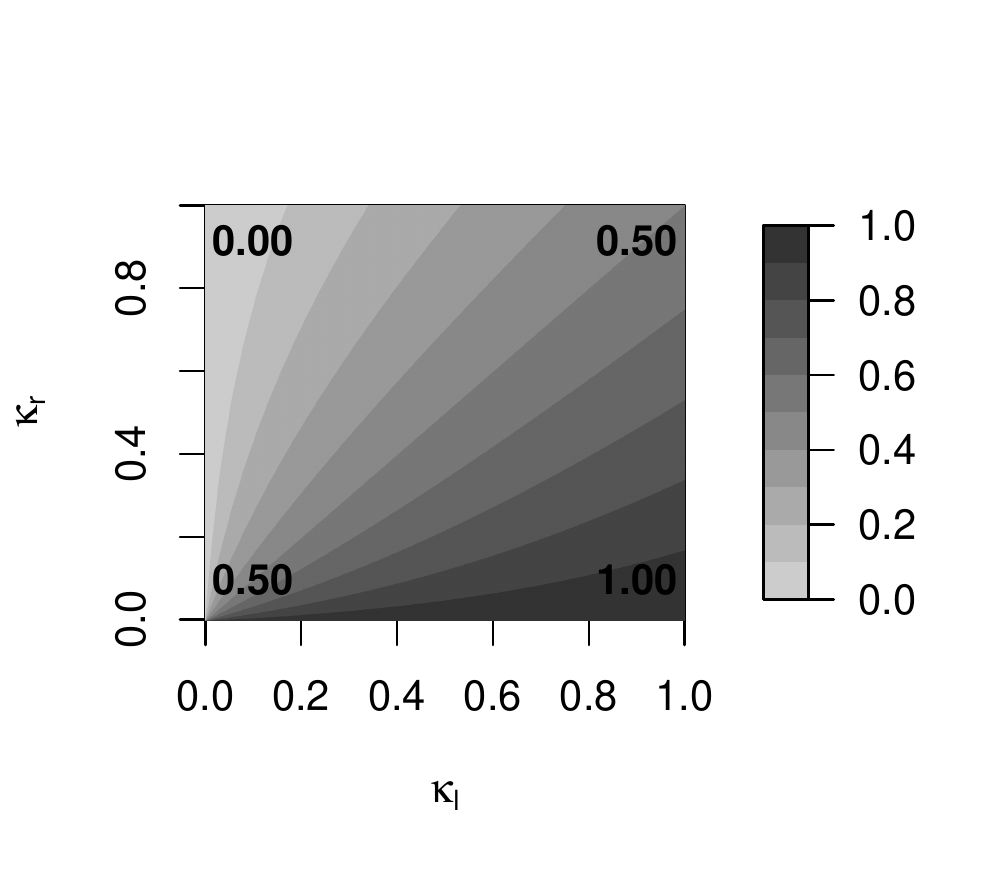}
		\caption{Mode.}
		\label{fig:BMTmode}
	\end{subfigure}
	
	\begin{subfigure}{0.32\textwidth}
		\includegraphics[width=\textwidth]{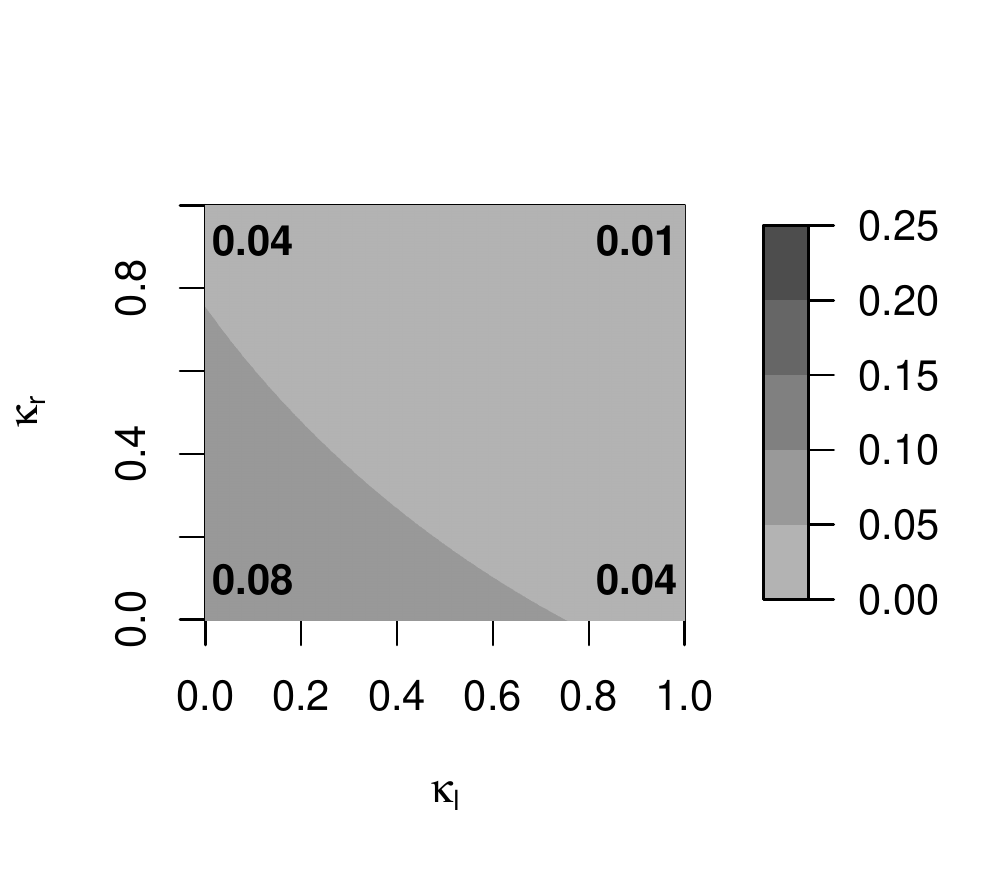}
		\caption{Variance.}
		\label{fig:BMTvar}
	\end{subfigure}
	\begin{subfigure}{0.32\textwidth}
		\includegraphics[width=\textwidth]{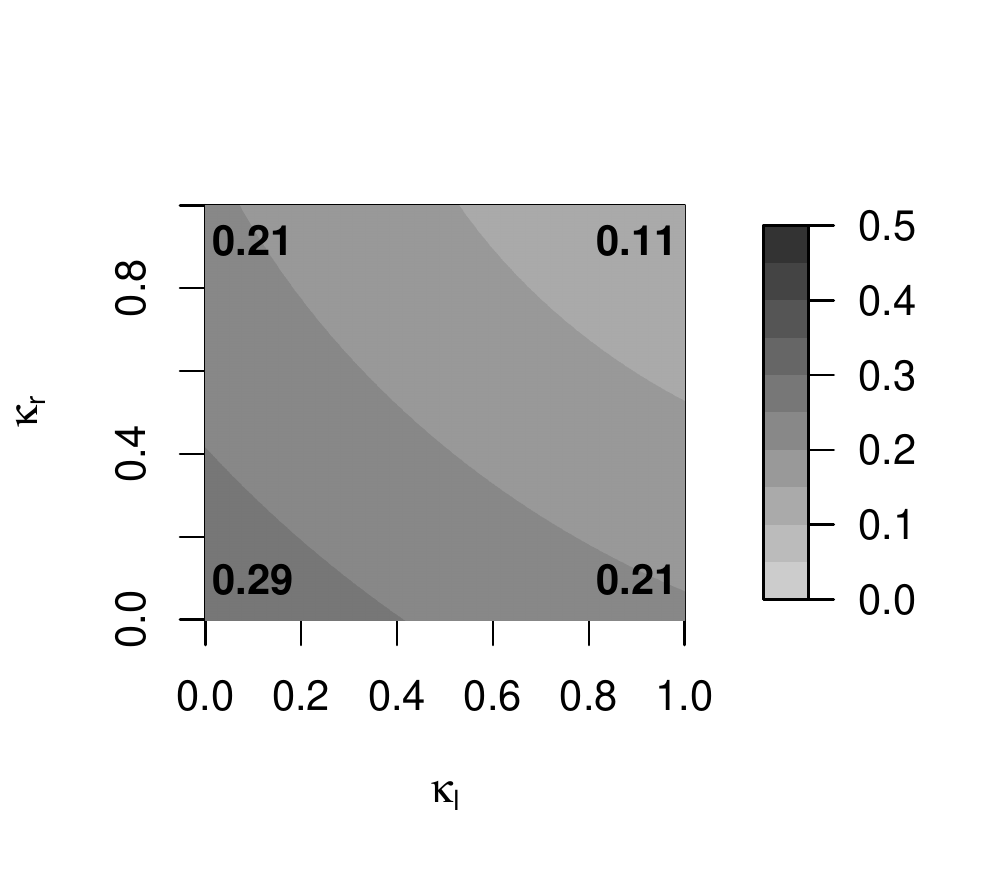}
		\caption{Standard deviation.}
		\label{fig:BMTsd}
	\end{subfigure}
	\begin{subfigure}{0.32\textwidth}
		\includegraphics[width=\textwidth]{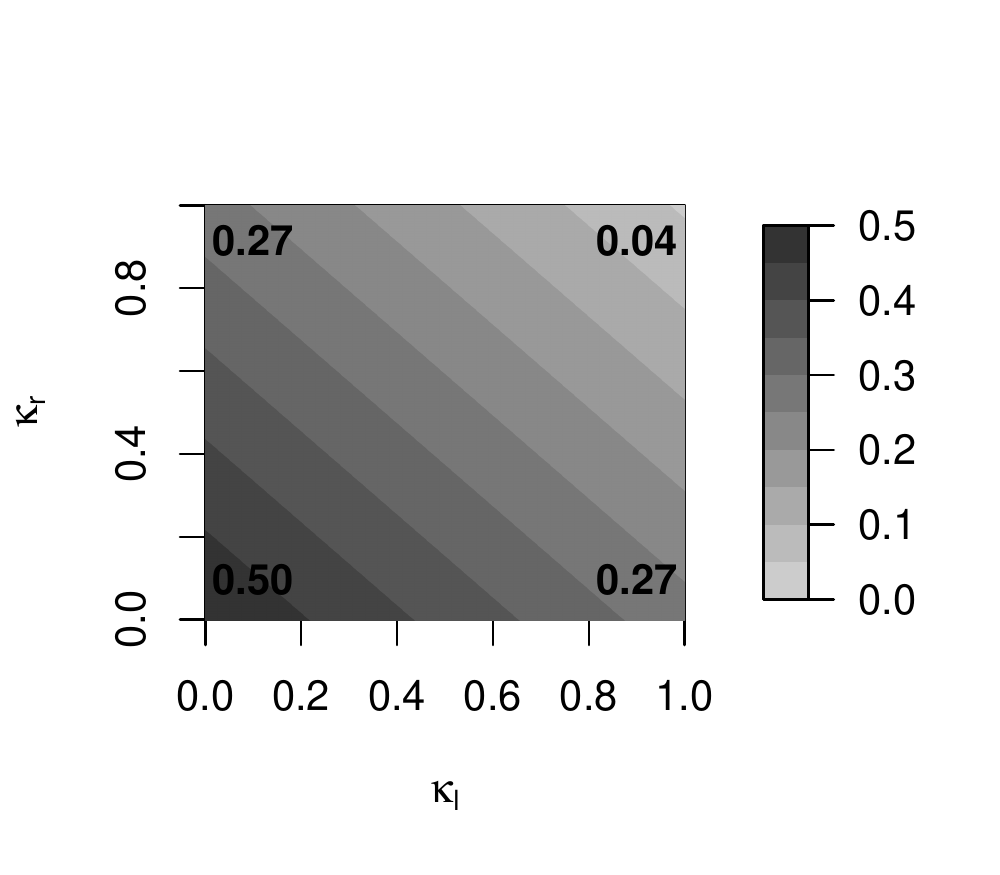}
		\caption{Interquartile range.}
		\label{fig:BMTiqr}
	\end{subfigure}
	
	\begin{subfigure}{0.32\textwidth}
		\includegraphics[width=\textwidth]{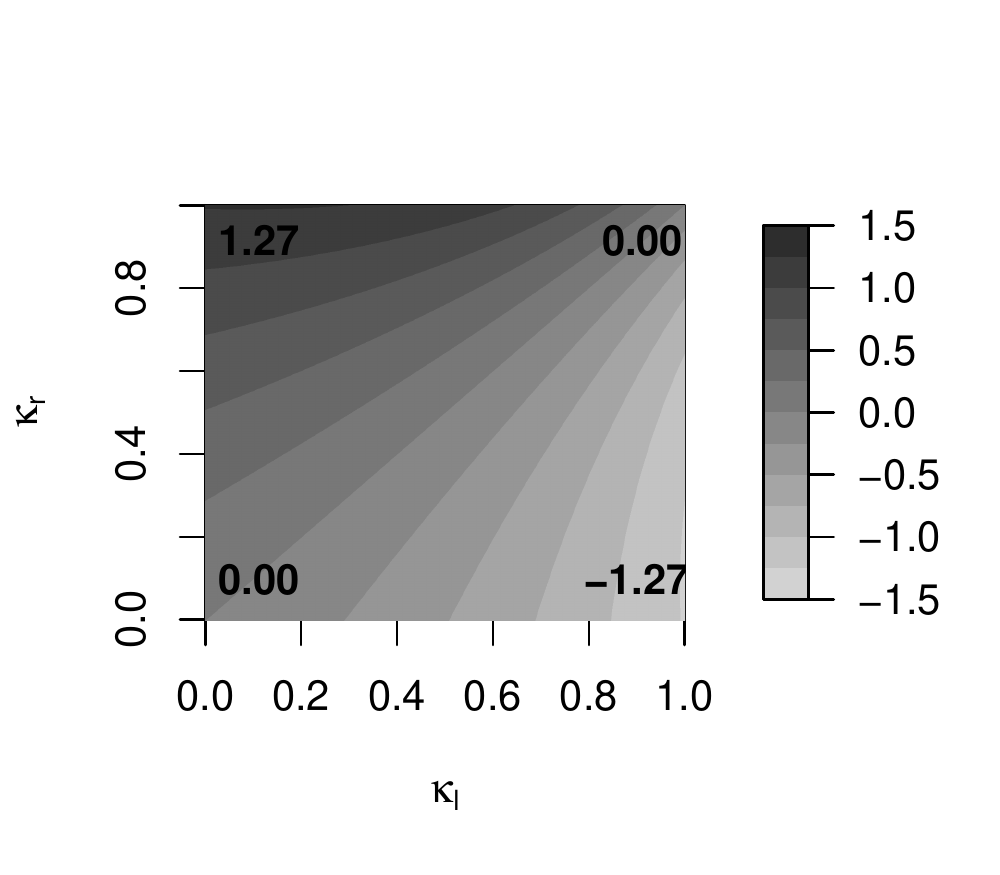}
		\caption{Pearson's skewness.}
		\label{fig:BMTskew}
	\end{subfigure}
	\begin{subfigure}{0.32\textwidth}
		\includegraphics[width=\textwidth]{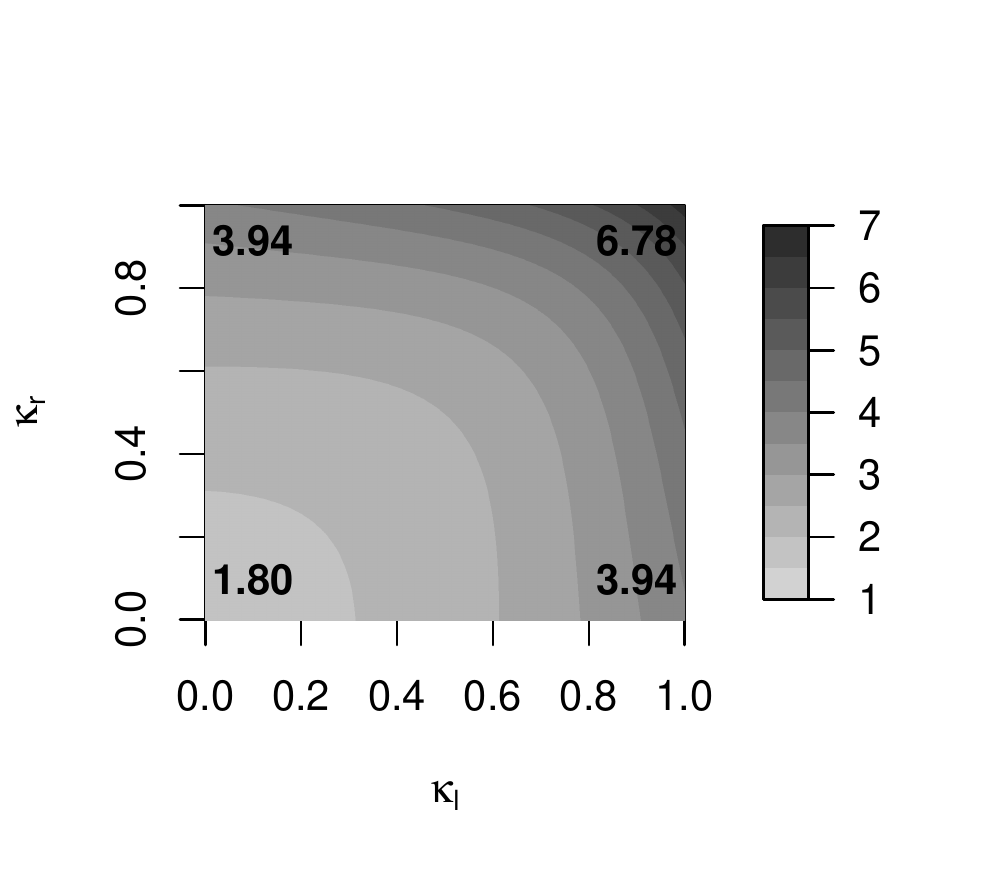}
		\caption{Pearson's kurtosis.}
		\label{fig:BMTkurt}
	\end{subfigure}
	\caption{Contour plots for some descriptive measures of the BMT distribution.}
	\label{fig:BMTmeasures}
\end{figure}

\newpage
\section{Results of simulations and parameter recovery} \label{ap:ParRecov}

\begin{table}[ht]
	\centering
	\caption{Mean, median, and standard deviation of the absolute difference between a parameter and its estimates for $1000$ samples, by sample size ($n$), parameter vector ($\boldsymbol{\theta}$), and estimation method (MLE: Maximum likelihood and MPSE: Maximum product of spacing).}
	\begin{tabular}{cclrrrrrr} \hline
		& & & \multicolumn{2}{c}{$\boldsymbol{\theta}_1 = \left(0.5,0.5\right)$} & \multicolumn{2}{c}{$\boldsymbol{\theta}_2 = \left(0.2,0.4\right)$} & \multicolumn{2}{c}{$\boldsymbol{\theta}_3 = \left(0.9,0.1\right)$} \\ \hline
		\multirow{6}{*}{n=30} & \multirow{3}{*}{MLE} & mean & 0.0980 & 0.1039 & 0.0902 & 0.1207 & 0.0863 & 0.0390 \\ 
		& & median & 0.0804 & 0.0848 & 0.0746 & 0.1071 & 0.0859 & 0.0319 \\ 
		& & sd & 0.0799 & 0.0811 & 0.0739 & 0.0911 & 0.0655 & 0.0327 \\ \cline{2-9}
		& \multirow{3}{*}{MPSE} & mean & 0.1099 & 0.1162 & 0.1040 & 0.1374 & 0.1015 & 0.0449 \\ 
		& & median & 0.0883 & 0.0978 & 0.0936 & 0.1176 & 0.0922 & 0.0396 \\ 
		& & sd & 0.0899 & 0.0898 & 0.0693 & 0.0997 & 0.0836 & 0.0319 \\ \hline
		\multirow{6}{*}{n=300} & \multirow{3}{*}{MLE} & mean & 0.0308 & 0.0316 & 0.0288 & 0.0369 & 0.0319 & 0.0123 \\ 
		& & median & 0.0259 & 0.0264 & 0.0249 & 0.0309 & 0.0268 & 0.0103 \\ 
		& & sd & 0.0230 & 0.0239 & 0.0220 & 0.0279 & 0.0244 & 0.0094 \\ \cline{2-9}
		& \multirow{3}{*}{MPSE} & mean & 0.0313 & 0.0321 & 0.0299 & 0.0375 & 0.0328 & 0.0127 \\ 
		& & median & 0.0261 & 0.0265 & 0.0250 & 0.0316 & 0.0279 & 0.0108 \\ 
		& & sd & 0.0236 & 0.0242 & 0.0224 & 0.0288 & 0.0248 & 0.0095 \\ \hline
		\multirow{6}{*}{n=3000} & \multirow{3}{*}{MLE} & mean & 0.0098 & 0.0095 & 0.0089 & 0.0115 & 0.0098 & 0.0041 \\ 
		& & median & 0.0083 & 0.0078 & 0.0074 & 0.0097 & 0.0083 & 0.0033 \\ 
		& & sd & 0.0073 & 0.0074 & 0.0070 & 0.0084 & 0.0076 & 0.0031 \\ \cline{2-9}
		& \multirow{3}{*}{MPSE} & mean & 0.0098 & 0.0095 & 0.0090 & 0.0115 & 0.0098 & 0.0041 \\ 
		& & median & 0.0082 & 0.0077 & 0.0075 & 0.0099 & 0.0082 & 0.0034 \\ 
		& & sd & 0.0073 & 0.0074 & 0.0070 & 0.0084 & 0.0076 & 0.0031 \\ \hline
	\end{tabular}
	\label{tab:ParRecov}
\end{table}	
								
\end{document}